\documentclass[reqno]{amsart}
\usepackage{hyperref}
\usepackage{amsfonts}
\usepackage{amsthm,amsmath}
\usepackage[integrals]{wasysym}
\usepackage{enumerate}
\usepackage{fancyhdr}
\usepackage{amssymb}
\usepackage{chemarrow}
\usepackage{tikz}
\usepackage{mathtools}
\usepackage{mathrsfs}%\AtBeginDocument{{\noindent\small
\usepackage{pdftexcmds}
\numberwithin{equation}{section}
\newtheorem{theorem}{Theorem}[section]

\newtheorem{definition}[theorem]{Definition}
\newtheorem{proposition}[theorem]{Proposition}
\theoremstyle{remark}
\newtheorem{remark}{Remark}
\allowdisplaybreaks

\def\div{ \hbox{\rm div}\,  }
\def\curl{ \hbox{\rm curl}\,  }

\def\R{{\mathbb R}}

\def\eps{\varepsilon}

\begin{document}
\title[\hfilneg \hfil ]
{Global weak solutions to the density-dependent Hall-magnetohydrodynamics system}

  \author[Tan]{ \text{Jin Tan}}
 \address[Jin Tan]{\newline Univ Paris Est Creteil,  CNRS, LAMA, F-94010 Creteil,  France \newline
 \& Univ Gustave Eiffel, LAMA, F-77447 Marne-la-Vall{\'e}e, France}
 \email{jin.tan@u-pec.fr}
 
%\thanks{}

\begin{abstract}
We are concerned with the global existence of  finite energy weak solutions to 3D  density-dependent magnetohydrodynamics (MHD) system with Hall-effect  set in a  general smooth bounded domain. The perfectly conducting wall boundary condition is imposed on the magnetic field.
%The fluid is supposed to be incompressible but with an inhomogeneous density, viscosity and electrical conductivity.
Due to the degeneracy of Hall-effect term (a tri-linear term) in  vacuum, we  assumed initial density lies in the bounded function space and having a positive lower bound. Particularly, these bounds are preserved by the density transport equation  which helps yield a satisfying compactness argument of  the magnetic field. % It turns out that the Hall-term  will not produce any disaster when applying weak convergence method
\end{abstract}

\maketitle
\section{Introduction and main results}
\subsection{Introduction}
 In this paper, we consider the following three dimensional density-dependent or inhomogeneous incompressible Magnetohydrodynamics system  that includes the Hall-effect (Hall-MHD):
\begin{align}
&\partial_t\rho+\div(\rho \textbf{u})=0\label{1.1},\\
&\partial_t{(\rho \textbf{u})}+\div(\rho \textbf{u}\otimes\textbf{u})-\div(2\mu d(\textbf{u}))+\nabla P= \curl \textbf{B}\times \textbf{B}\label{1.2},\\
&\div\textbf{u}=0\label{1.3},\\
&\partial_t{ \textbf{B}}+\curl(\textbf{B}\times \textbf{u}+{h}\dfrac{\curl \textbf{B}\times \textbf{B}}{\rho})=-\curl(\dfrac{\curl \textbf{B}}{\sigma})\label{1.4}. 
\end{align}
The unknowns are the density of the fluid $\rho$, the fluid velocity $\textbf{u}\in\R^3$, the magnetic field $\textbf{B}\in\R^3$ and the scalar pressure $P.$ We denote by $d(\textbf{u})=\frac12(\nabla\textbf{u}+\nabla\textbf{u}^T)$ the shear rate tensor,  $\mu=\mu(\rho)$ the fluid viscosity and $\sigma=\sigma(\rho)$ its electrical conductivity (this dependence  enables us to consider the motion of several immiscible fluids with various viscosities and conductivities), both conductivity and viscosity being positive continuous functions on the $[0, \infty)$ . The dimensionless number ${h}$ measures   the 
magnitude  of the so-called Hall-effect term: $\curl(h\dfrac{\curl \textbf{B}\times \textbf{B}}{\rho}),$ compared to  the typical length scale of the fluid.

The above system is used to model the evolution of electrically conducting fluids such as 
plasmas or electrolytes (then, ${\bf u}$ represents the ion velocity) and takes into account 
 the fact that in a moving conductive fluid,  the magnetic field can induce currents which, in turn, polarize the fluid and change the magnetic field. In the work of Acheritogaray, Degond, Frouvelle and Liu  \cite{{MR2861579}}, they derived the following generalized Ohm's law from the two-fluids Navier-Stokes-Maxwell model under suitable scaling hypotheses:
$$\textbf{j}=\sigma(\textbf{E}+\textbf{u}\times\textbf{B}+\nabla(\ln \rho)-h\dfrac{\textbf{j}\times\textbf{B}}{\rho}),$$
where $\textbf{j}=\curl\textbf{B}$ is the current density and $\textbf{E}$ the electric field. Jang and Masmoudi in \cite{Ja12} also gave a derivation of Hall-effect.
The Maxwell-Farady equations:
$$\partial_t \bf{B}+\curl \bf{E}=0$$
with above generalized Ohm's law then gives rise to \eqref{1.4}.
  Compared with the classical inhomogeneous incompressible MHD equations, the  density-dependent Hall-MHD system have an additional Hall-effect term which is believed to be the key for understanding the problem of magnetic reconnection, as observed in space plasmas, star formation, neutron stars and geo-dynamo  (see for example \cite{{Braiding2004}, Fo91, Ho05, Hu03}). 
Meanwhile, since Hall-effect term is a  tri-linear term and degenerate in vacuum, it makes  the mathematical analysis of the density-dependent Hall-MHD system  much more complicated.

\smallbreak
The primary aim of this paper is to establish the existence of global weak solutions that could be called $``\mathbf{solutions}~ \grave{\mathbf a}~ \mathbf{la~ Leray}"$ to the density-dependent Hall-MHD system by analogy with the classical global existence results for the incompressible homogeneous Navier-Stokes equations obtained by Leray \cite{JL34}
and the density-dependent Navier-Stokes equations by Lions \cite{PL96}. 
Without taking into consideration of Hall-effect term (i.e. $h=0$), many works have already been  addressed. For instance, with constant density, global existence for standard viscous resistive incompressible MHD system has been previously proved by Duvaut and Lions \cite{DL72}, see also Sermange and Temam \cite{ST83},  with inhomogeneous density Gerbeau and Le Bris \cite{GB97} proved existence of a global weak solution. Consider incompressible homogeneous Hall-MHD system,  the global existence of Leray-Hopf weak solutions to the periodic case is first studied  in \cite{MR2861579}, later, Chae et al. \cite{MR3208454} have treated $\R^3$ case  as well as the local well-posedness of classical solutions with initial data in regular Sobolev spaces. Weak-strong uniqueness and energy identity have been investigated by  Dumas and Sueur in \cite{MR3227510}.  Partial regularity have been studied by  Chae and Wolf in a series works \cite{Chae-P-1, Chae-P-2, Chae-P-3} and  later by Zeng and Zhang \cite{ZZ}. Very recently,  Danchin and the author in \cite{D-Besov,D-Sobolev}  have first established existence results for initial data with critical regularity in Besov spaces and Sobolev spaces, Dai \cite{Dai19} shows  a non-uniqueness result for weak solutions having Leray-Hopf type regularity.
And Zhang \cite{Zh20}  proved existence of  weak solutions that corresponds to Fujita-Kato solutions to the 3D inhomogeneous incompressible Navier-Stokes equations.

As far as we know, there are few  existence results for the  Hall-MHD system in a general bounded domain. Unlike  MHD system in a bounded domain,  the appearance of the Hall-effect term gives us  additional difficulties in non-linear analysis and on boundary  condition for magnetic field. Let $\Omega $ be a smooth, bounded, fixed connected open subset  of $\R^3.$ We shall denote by $\textbf{n}$ the outward-pointing normal to $\Omega.$
For the boundary conditions, we consider the no-slip boundary condition for the fluid velocity $\textbf{u}$ :
\begin{equation}\label{bcu}
\textbf{u}=0  \hspace*{0.5cm}~~{\rm{on}}~~\partial\Omega,
\end{equation}
the perfectly conducting wall boundary condition for the $\textbf{B}$ field (see \cite{MR2861579}) :
\begin{equation}\label{bcb}
\left\{
\begin{aligned}
&\mathbf B\cdot \mathbf n=0\hspace*{4cm}~~{\rm{on}}~~\partial\Omega,\\
&(h\dfrac{\curl\textbf{B}\times\textbf{B}}{\rho}+\dfrac{\curl\textbf{B}}{\sigma})\times\mathbf n=0 \hspace*{0.5cm}~~{\rm{on}}~~\partial\Omega,
\end{aligned}
\right.
\end{equation}
that is assuming zero normal component of the magnetic field and zero tangential component of the electric field. This non-linear boundary condition has been  emphasized by Chae et al. in \cite{MR2861579, {MR3186849}} for homogeneous case. Indeed,  it seems hard to study Hall-MHD and related system in a general bounded domain since this non-linear boundary condition.  To our best knowledge,  the only well-posedness result of strong solutions with perfectly conducting wall boundary condition is due to Mulone and Solonnikov \cite{MS95}.
Recently, Han et al. \cite{Han1}  considered slip  boundary condition for $\bf{u}$ and no-slip boundary condition for $\bf{B}$  and established global weak solutions to the Hall-MHD system with the ion-slip effect in a bounded domain. Later, Han and Hwang \cite{Han2} imposed a new boundary condition and proved local well-posedness of strong solutions with a regularity criteria.
However, as we only consider solutions defined in the weak sense,  the non-linear boundary condition is avoided.

Due to the degeneracy of Hall-effect term in  vacuum, we shall suppose in this paper that initial density does not vanish. Thanks to the nature of \eqref{1.1} with \eqref{1.3} this property will be preserved along the time and it is crucial for our later analysis. If there exists vacuum, we have no idea how to make sure the Hall-effect term is well-defined. Let us look at the compressible Hall-magnetohydrodynamics system in e.g. \cite{Fa15, Xiang17}  for a while, 
\begin{align*}
&\partial_t\rho+\div(\rho \textbf{u})=0,\\
&\partial_t{(\rho \textbf{u})}+\div(\rho \textbf{u}\otimes\textbf{u})-\mu\Delta\mathbf{u}+\nabla P= \curl \textbf{B}\times \textbf{B},\\
&\partial_t{ \textbf{B}}+\curl(\textbf{B}\times \textbf{u}+h\dfrac{\curl \textbf{B}\times \textbf{B}}{\rho})=\Delta\mathbf{B}, \\
&\div \textbf{B}=0.
\end{align*}
The celebrated results of Lions \cite{PL98} and extended by Fereisl \cite{Fe03} on the weak solutions for compressible Navier-Stokes equations may not be successfully applied to the above model, since without the incompressibility condition \eqref{1.3} on the velocity field we are not able to control vacuum regions even if we assume there is no vacuum at beginning. 
Still, global low-energy weak solutions of  3D compressible MHD equations with density positive and essentially bounded were established in \cite{Suen12}. This motivates us to study the 
existence of global weak solutions of the above Compressible Hall-MHD system with initial velocity small in $L^2(\mathbb{R}^3).$

\subsection{A priori estimate and functional spaces}
Different from the work of Gerbeau and Le Bris, where they only assumed positive initial density,  we need to assume  initial density having a positive lower bound for technical reason (see Remark. \ref{Re-low-bound-rho}). Given that $\inf \rho_0 >0,$ it is reasonable to make an initial hypothesis on velocity not momentum.  We thus impose the following initial conditions:
\begin{align}
&\rho|_{t=0}=\rho_0\geq \underline{\rho}>0 \hspace*{0.5cm}~~{\rm{on}}~~\Omega\label{ind},\\
&\textbf{u}|_{t=0}=\textbf{u}_0 \hspace*{1.7cm}~~{\rm{on}}~~\Omega\label{inm},\\
&\textbf{B}|_{t=0}=\textbf{B}_0  \hspace*{1,6cm}~~{\rm{on}}~~\Omega\label{inma}.
\end{align}
And we shall suppose in the sequel that for $\xi\in[\underline{\rho}, \infty),$
\begin{equation}\label{positivemu}
0<\underline{\mu}\leq\mu(\xi)\leq \bar\mu,
\end{equation}
\begin{equation}\label{positivesigma}
0<\underline{\sigma}\leq\sigma(\xi)\leq\bar{\sigma}.
\end{equation}
Now, we formally derive an a priori energy estimate. We first remark that \eqref{1.1} and the incompressibility condition \eqref{1.3} immediately imply that
\begin{equation*}
\underline{\rho}\leq\rho(t, x)\leq\|\rho_0\|_{L^\infty},\quad{\rm{a.e.\,\,\mathit{t},\,\,\mathit{x}}}.
\end{equation*}
Next, recall that for all velocity fields $\textbf{u}$ and density $\rho$
 \begin{equation}\label{v1}
 \div(\rho\textbf{u}\otimes \textbf{u})=\textbf{u}\,\div(\rho\textbf{u})+\rho\textbf{u}\cdot\nabla\textbf{ u},
 \end{equation}
 in the sense of distributions on $\Omega.$
Moreover, we shall make frequent use of the following formula of vector analysis : for all vector fields $\textbf{v}$ and $\textbf{w}$ in $\R^3$ we have
\begin{equation}\label{v2}
\int_{\Omega} \curl \textbf{v}\cdot \textbf{w}\,dx = \int_{\partial\Omega} (\textbf{n}\times \textbf{v})\cdot \textbf{w}\,dS+\int_{\Omega} \textbf{v}\cdot\curl  \textbf{w}\,dx,
\end{equation}
whenever these integrals make sense. Here $dS$ is the standard surface measure of $\partial\Omega.$

Thanks to \eqref{1.1} and \eqref{v1}, we multiply \eqref{1.2} by $\textbf{u},$  integrate over $\Omega$ and use boundary  condition \eqref{bcu} to get
\begin{equation}\label{enu}
\frac12\dfrac{d}{dt}\int_{\Omega}\rho|\textbf{u}|^2\,dx+\frac12\int_{\Omega}\mu(\rho)|\nabla\textbf{u}+\nabla\textbf{u}^T|^2\,dx
=\int_{\Omega}(\curl \textbf{B}\times \textbf{B})\cdot \textbf{u}\,dx.
\end{equation}
 Multiplying \eqref{1.4} by $\textbf{B}$, integrating over $\Omega,$ using \eqref{v2}, we have
 \begin{align*}
 &\frac12\dfrac{d}{dt}\int_{\Omega}|\textbf{B}|^2\,dx+\int_{\Omega}\bigl(\textbf{B}\times \textbf{u}+h\dfrac{\curl\textbf{B}\times \textbf{B}}{\rho}+\dfrac{\curl\textbf{B}}{\sigma(\rho)}\bigr)\cdot\curl\textbf{B}\,dx\\
 &\quad=-\int_{\partial\Omega}\bigl(\textbf{n}\times\bigl(\textbf{B}\times \textbf{u}+h\dfrac{\curl\textbf{B}\times \textbf{B}}{\rho}+\dfrac{\curl\textbf{B}}{\sigma(\rho)}\bigr)\bigr)\cdot\textbf{B}\,dS.
 \end{align*}
Thus, using the facts that
$$(\curl\textbf{B}\times\textbf{B})\cdot\curl\textbf{B}=0,$$
$$(\textbf{B}\times\textbf{u})\cdot\curl\textbf{B}=(\curl\textbf{B}\times\textbf{B})\cdot\textbf{u},$$
and boundary conditions \eqref{bcu}, \eqref{bcb} we get
 \begin{equation}\label{enb}
\frac12\dfrac{d}{dt}\int_{\Omega}|\textbf{B}|^2\,dx+\int_{\Omega}\dfrac{|\curl\textbf{B}|^2}{\sigma}\,dx=-\int_{\Omega}(\curl\textbf{B}\times\textbf{B})\cdot\textbf{u}\,dx.
 \end{equation}
Adding up \eqref{enu} and \eqref{enb}, we formally have the following energy equality 
\begin{equation}\label{en1}
\frac12\dfrac{d}{dt}\bigl(\int_{\Omega}\rho|\textbf{u}|^2\,dx+\int_{\Omega}|\textbf{B}|^2\,dx\bigr)+\frac12\int_{\Omega}\mu(\rho)|\nabla\textbf{u}+\nabla\textbf{u}^T|^2\,dx+\int_{\Omega}\dfrac{|\curl\textbf{B}|^2}{\sigma(\rho)}\,dx=0.
\end{equation}
It is well-known and wide open in \cite{JL34, PL96}, when dealing with global weak solutions of non-linear partial differential equations, the global weak solutions we obtained usually satisfy the energy inequality instead of energy equality 
\begin{align}
&\frac12\int_{\Omega}(\rho|\textbf{u}|^2+|\textbf{B}|^2)\,dx+\frac12\int_0^T\int_{\Omega}\mu(\rho)|\nabla\textbf{u}+\nabla\textbf{u}^T|^2\,dxdt+\int_0^T\int_{\Omega}\dfrac{|\curl\textbf{B}|^2}{\sigma(\rho)}\,dxdt\notag\\
&\quad\leq\frac12\int_{\Omega} \bigl({\rho_0}{|\textbf{u}_0|^2}+|\textbf{B}_0|^2\bigr)\,dx\label{en2}.
\end{align}

 Let us notice from \eqref{positivemu} that
\begin{equation*}
\int_{\Omega}\mu(\rho)|\nabla\textbf{u}+\nabla\textbf{u}^T|^2\,dx\geq 2\underline{\mu}\int_{\Omega}|\nabla\textbf{u}|^2\,dx,
\end{equation*}
since $\div\textbf{u}=0.$
Then thanks to \eqref{positivesigma} we have 
\begin{align}
&\frac12\int_{\Omega}(\underline{\rho}|\textbf{u}|^2+|\textbf{B}|^2)\,dx+\underline{\mu}\int_0^T\int_{\Omega}|\nabla\textbf{u}|^2\,dxdt+\frac{1}{\bar\sigma}\int_0^T\int_{\Omega}|\curl\textbf{B}|^2\,dxdt\notag\\
&\quad\leq\frac12\int_{\Omega}\bigl( \rho_0{|\textbf{u}_0|^2}+|\textbf{B}_0|^2\bigr)\,dx\label{en3}.
\end{align}

 In order to formulate our problem and the main results, let us recall the definition of some functional spaces that we shall use throughout this paper. The space $\mathcal{D}(\Omega)$ is defined as the space of smooth functions compactly supported in the domain $\Omega,$ and $\mathcal{D}^{'}(\Omega)$ as the space of  distributions on $\Omega$. For $X$ a Banach space, $p\in[1, \infty]$ and $T>0$, the notation $L^p(0, T; X)$ or $L^p_T(X)$ designates the set of measurable functions $f: [0, T]\to X$ with $t\mapsto\|f(t)\|_X$ in $L^p(0, T)$, endowed with the norm $\|\cdot\|_{L^p_{T}(X)} :=\|\|\cdot\|_X\|_{L^p(0, T)},$ and agree that $\mathcal C([0, T]; X)$ denotes the set of continuous functions from $[0, T]$ to $X$. Slightly abusively, we will keep the same notations for multi-component functions.
 The space $H^1(\Omega)$ denotes the space of $L^2$ functions of $f$ on $\Omega$ such that $\nabla f$ also belongs to $L^2(\Omega).$ The Hilbertian norm is defined by 
 $$\|f\|_{H^1(\Omega)}^2 :=\|f\|_{L^2(\Omega)}^2+\|\nabla f\|_{L^2(\Omega)}^2.$$
 The space $H^1_0(\Omega)$ is defined as the closure of $\mathcal{D}(\Omega)$ for the $H^1(\Omega)$ norm, and the space $H^{-1}(\Omega)$ as the dual space of $H^1_0(\Omega)$ for the $\mathcal{D}^{'}\times\mathcal{D}$ duality. Then we introduce the vector-valued spaces
\begin{align*}
\mathcal{V}&=\{\textbf{u}\in\mathcal{D}(\Omega), ~\div\textbf{u}=0\},\\
\mathcal{H}&={\rm{the~closure~of}}~\mathcal{V}~{\rm{in}}~L^2(\Omega)\\
&=\{\mathbf{u}\in{L^2(\Omega)},~\div\mathbf{u}=0,~\textbf{u}\cdot\textbf{n}|_{\partial\Omega}=0\},\\
V_1&={\rm{the~closure~of}}~\mathcal{V}~{\rm{in}}~H^1(\Omega)\\
&=\{\mathbf{u}\in{H^1_0(\Omega)},~\div\mathbf{u}=0\},\\
\mathcal W&=\{\textbf{B}\in\mathcal{C}^\infty(\bar\Omega), ~\div\textbf{B}=0, ~\textbf{B}\cdot\textbf{n}|_{\partial\Omega}=0\},\\
W_1&={\rm{the~closure~of}}~\mathcal{W}~{\rm{in}}~H^1(\Omega)\\
&=\{\mathbf{B}\in{H^1(\Omega)},~\div\mathbf{B}=0,~\textbf{B}\cdot\textbf{n}|_{\partial\Omega}=0\},
\end{align*}
and their norms
\begin{align*}
&\|\textbf{u}\|_{V_1}=\|\nabla\textbf{u}\|_{L^2(\Omega)},\\
&\|\textbf{B}\|_{W_1}=\|\curl\textbf{B}\|_{L^2(\Omega)}.
\end{align*}
We equip $H$ with the following inner product
$$(\mathbf{u},\mathbf{v})=\int_{\Omega}\mathbf{u}\cdot\mathbf{v}\,dx,~{\rm{for~all}}~\mathbf{u}, \mathbf{v}\in H.$$
Just remark that, one can establish that $\|\cdot\|_{V_1}$ and $\|\cdot\|_{W_1}$ defined two norms which are equivalent to that introduced by  $H^1(\Omega)$ on $V_1$ and $W_1,$ respectively (cf. Duvaut and  Lions \cite{DL72},  Sermange and  Temam \cite{ST83}).

In accordance with \eqref{1.3},  we assume that $\div {\bf u}_0=0$ 
and, for physical consistency,  since a magnetic field has to be 
divergence free,  we suppose  that $\div {\bf B}_0=0,$ too,
a property that is conserved  through the evolution. 
With above notations and a priori estimate \eqref{en3}, we propose the following assumptions on the initial data:
\begin{align}
&\rho_0\in L^\infty, \quad\rho_0\geq\underline{\rho}>0\label{inassum-rho},\\
&\textbf{u}_{0}\in \mathcal{H},\label{inassum-u}\\
&\textbf{B}_{0}\in \mathcal{H}\label{inassum-b}.
\end{align}

In a same fashion with \cite{PL96, GB97}, we define our weak solutions as follows.
\begin{definition}\label{defweaksolu}
We say that $(\rho, \mathbf{u}, \mathbf{B})$ is a global weak solution of the problem \eqref{1.1}-\eqref{inma} with the initial assumptions \eqref{inassum-rho}-\eqref{inassum-b}, if for any $T>0,$ $(\rho, \mathbf{u}, \mathbf{B})$ satisfies the following properties:
\begin{align*}
&\rho\geq\bar\rho,\quad\rho\in L^\infty((0, T)\times \Omega), \quad\rho\in \mathcal{C}([0, T]; L^p(\Omega)),\quad1\leq p<\infty,\\
&\mathbf{u}\in L^2(0, T; V_1),\quad{\rm{and}}\quad\rho|\mathbf{u}|^2\in L^\infty(0, T; L^1({\Omega})),\\
&\mathbf{B}\in L^2(0, T; W_1)\cap L^\infty(0, T; L^2(\Omega)),
\end{align*}
moreover, for any  $\phi\in\mathcal{C}^1([0, T]\times\Omega)$ with $\phi(T, \cdot)=0,$
\begin{equation}\label{wfdensity}
-\int_0^T\int_{\Omega}(\rho\partial_t \phi+\rho\mathbf{u}\cdot\nabla\phi)\,dxdt=\int_{\Omega}\rho_0\phi(0, x)\,dx;
\end{equation}
for any $\Phi\in\mathcal{C}^1([0, T]\times\Omega)$ with $\div\Phi=0$ and $\Phi(T, \cdot)=0,$
\begin{align}
&\int_0^T\int_{\Omega}(-\rho\mathbf{u}\cdot\partial_t\Phi-(\rho\mathbf{u}\otimes\mathbf{u})\cdot\nabla\Phi+2\mu(\rho) d(\mathbf{u})\cdot d(\Phi))\,dxdt\notag\\
&\quad=\int_0^T\int_{\Omega}(\curl\mathbf{B}\times\mathbf{B})\cdot\Phi\,dxdt+\int_{\Omega}\mathbf{m}_0\cdot\Phi(0, x)dx\label{wfu};
\end{align}
for any $\Psi\in\mathcal{C}^1([0, T]\times\Omega)$ with $\Psi(T, \cdot)=0,$
\begin{align}
&\int_0^T\int_{\Omega}\bigl(-\mathbf{B}\cdot\partial_t\Psi+(\mathbf{B}\times\mathbf{u}+h\dfrac{\curl\mathbf{B}\times\mathbf{B}}{\rho}+\dfrac{\curl\mathbf{B}}{\sigma(\rho)})\cdot\curl\Psi\bigr)\,dxdt\notag\\
&\quad=\int_{\Omega}\mathbf{B}_0\cdot\Psi(0, x)dx\label{wfb};
\end{align}
and finally, the energy inequality \eqref{en2} holds for all $t\in[0, T].$
\end{definition}

\subsection{Main results}
In comparison with the models studied in \cite{PL96, GB97}, the main difficulty of proving the existence of weak solutions lies in the Lorentz force term $\curl\textbf{B}\times \textbf{B}$ in the Navier-Stokes equations \eqref{1.2} while $\textbf{B}$ satisfies a quasi-linear parabolic equations with a tri-linear term involving density, we have to recover some compactness on $\textbf{B}$  in order to pass to the limit in the non-linear terms.

We now state our main theorem.

\begin{theorem}\label{Th_1}
Under the regularity assumptions \eqref{inassum-rho}-\eqref{inassum-b} on the initial data, the initial-boundary value problem to the density-dependent Hall-MHD system \eqref{1.1}-\eqref{inma} has a weak solution in the sense of  Definition \ref{defweaksolu}. Furthermore, we have for all $0\leq\alpha\leq\beta\leq\infty$
\begin{align}
{\rm{meas}}\{x\in\R^3,~ \alpha\leq\rho(t,x)\leq\beta\}\quad{\rm{is~independent~of}}~t\geq0, \label{meas}
\end{align}
and\footnote{The space $\mathcal{C}_w([0, T]; L^2(\Omega))$  denotes continuity on the interval $[0, T]$ with values in the weak topology of $L^2(\Omega)$}
$$\mathbf{u},~\mathbf{B}\in\mathcal{C}_w([0, T]; L^2(\Omega)).$$
\end{theorem}
%\begin{remark}
%As we did not assume any continuity on $\rho\mathbf{u}$ in Definition \ref{defweaksolu}, therefore the sense of the initial condition \eqref{initialmomentum} is not clear as remarked in \cite{PL96}. However, if $\mathbf{u}(0, x)$ is divergence free, then we can prove that $\mathbf{u}\in\mathcal{C}_w([0, T]; L^2(\Omega)),$ which together with the continuity of $\rho$  gives sense to \eqref{initialmomentum}.
%\end{remark}
The next theorem states a weak-strong uniqueness property of the solution. We show that any global weak solution coincides with a more regular solution as long as the latter exists.
\begin{theorem}\label{Th_2}
Assume $\mu, \sigma$ are locally Lipschitz continuous. Let $(\rho, \mathbf{u}, \mathbf{B})$ be a weak solution obtained from Theorem \ref{Th_1}. If there exists a  solution $(\hat\rho, \mathbf{\hat u}, \mathbf{\hat B})\in\mathcal{C}([0, T]\times\Omega)$ of the problem \eqref{1.1}-\eqref{bcb}, which satisfies
\begin{align}
\nabla\hat\rho, ~\partial_t\mathbf{\hat u}, ~\partial_t\mathbf{\hat B}\in L^2(0, T; L^3(\Omega))\quad{\rm{and}}\quad\nabla\mathbf{\hat u}, \nabla\mathbf{\hat B}\in L^2(0, T; L^\infty(\Omega)),\label{regularity-ws}
\end{align}
and at initial time
\begin{align*}
\hat\rho|_{t=0}=\rho_0,\quad\mathbf{\hat u}|_{t=0}=\mathbf{u}_0,\quad\mathbf{\hat B}|_{t=0}=\mathbf{B}_0 \quad{\rm{in}}~\Omega,
\end{align*}
then we have $(\rho, \mathbf{u}, \mathbf{B})\equiv(\hat\rho, \mathbf{\hat u}, \mathbf{\hat B})$ a.e. in $(0, T)\times\Omega.$
\end{theorem}
In the rest of this paper, we first set up our approximation scheme and establish the existence of solutions to the approximation problem in Section \ref{sec.2}. Then in Section \ref{sec.3}, in order to recover the original system, we deduce some compactness results and finally finish the proof of our main theorem. In the end, we prove Theorem \ref{Th_2}.

Throughout this paper, we use $C$ to denote a general positive constant which may
different from line to line.

%%%%%%%%%%%%%%%%%%%%%%%%%%%%%%%%%%%%%%%%%%%%%%%%%%%%%%%%%%%%%%%%%%%%%%%%
\section{Approximation Scheme}\label{sec.2}
The essential idea to prove the existence of a weak solution to the density-dependent Hall-MHD system is to introduce an approximation problem, that allows one to define \eqref{1.1} as a classical transport equation. Now, presenting the approximation problem and showing the existence of a regular solution to this problem is our purpose for the next two steps.
\subsection{First step : a linearized  problem}
At this step, we prove a preliminary result for a linearized  problem with prescribed density, which will be useful in subsection \ref{sec2.2}.
For this purpose, we first define two finite dimensional spaces for $n\in\mathbb{N}^* :$ 
$$\mathcal{V}_{n}={\rm{span}}\{\Theta_i\}_{i=1}^n,\quad{\rm{and}}\quad\mathcal{W}_{n}={\rm{span}}\{\Gamma_i\}_{i=1}^n,$$
where $\{\Theta_i\}_{i=1}^\infty\subset\mathcal{V}$ and $\{\Gamma_i\}_{i=1}^\infty\subset\mathcal{W}$ are  orthonormal basis of $V_1$ and $W_1$ respectively. 

For $\rho$, $\mathbf{v}$, $\mathbf{H}$ and $n,$ $T>0$ arbitrarily fixed such that 
\begin{align}
&\rho,~\partial_t\rho\in\mathcal{C}([0, T], \mathcal{C}^1(\bar\Omega)),~{\rm{and~0<\rho_1\leq\rho(t, \mathnormal{x})\leq\rho_2,}}\label{fixedrho}\\
&\mathbf{v}\in \mathcal{C}([0, T]; \mathcal{V}_n)~{\rm{such~that}}~\partial_t\rho+\div(\rho\mathbf{v})=0\label{lv},\\
&\mathbf{H}\in\mathcal{C}([0, T]; \mathcal{W}_n)\label{lb},
\end{align}
the linearized  problem is to find a couple of  vector-valued functions $(\mathbf{u}, \mathbf{B})$ such that 
\begin{equation}\label{lwfu}
\int_{\Omega}\rho(\partial_t\mathbf{u}+\mathbf{v}\cdot\nabla\mathbf{u})\cdot\Phi\,dx+\int_{\Omega}2\mu(\rho) d(\mathbf{u})\cdot d(\Phi)\,dx=\int_{\Omega}(\curl\mathbf{B}\times\mathbf{H})\cdot\Phi\,dx,
\end{equation}for any $\Phi\in\mathcal{V}_n,$ 
\begin{equation}\label{lwfb}
\int_{\Omega}\partial_t\mathbf{B}\cdot\Psi\,dx+\int_{\Omega}\Bigl(\mathbf{H}\times\mathbf{u}+h\dfrac{\curl\mathbf{B}\times\mathbf{H}}{\rho}+\dfrac{\curl\mathbf{B}}{\sigma(\rho)}\Bigr)\cdot\curl\Psi\,dx=0,
\end{equation}for any $\Psi\in\mathcal{W}_n,$
and with initial condition $(\mathbf{u}|_{t=0}, \mathbf{B}|_{t=0})=(\mathbf{u}_0, \mathbf{B}_0).$

To show a existence result for the above linearized problem is not difficult, it will  based on the Galerkin's method. We have
\begin{proposition}\label{P.l}
Let $(\mathbf{u}_0, \mathbf{B}_0)\in\mathcal{V}_n\times\mathcal{W}_n.$ Under the assumptions \eqref{fixedrho}-\eqref{lb}, there exists a unique pair of solution $(\mathbf{u}, \mathbf{B})\in\mathcal{C}^1([0, T]; \mathcal{V}_n)\times\mathcal{C}^1([0, T]; \mathcal{W}_n)$ to the problem \eqref{lwfu}-\eqref{lwfb}. Moreover, we have energy equality \eqref{en1}.
\end{proposition}
\begin{proof}
We look for a solution $(\mathbf{u}, \mathbf{B})$  under the form
\begin{equation}\label{lub}
\left\{
\begin{aligned}
&\mathbf{u}=\sum_{i=1}^n\alpha_{i}(t)\Theta_i,\\
&\mathbf{B}=\sum_{i=1}^n\beta_{i}(t)\Gamma_i.
\end{aligned}
\right.
\end{equation}
Look at weak form \eqref{lwfu} and \eqref{lwfb}, replace $\Phi$ by $\Theta_j$ and $\Psi$ by $\Gamma_j$ for $j=1,\dots,n,$ respectively, we find that scalar functions $\alpha_i$ and $\beta_i$ $(i=1,\dots,n)$ are solutions of the following linear ODEs :
\begin{equation}\label{ODEs}
\left\{
\begin{aligned}
& \Bigl(\int_{\Omega}\rho\Theta_i\cdot\Theta_j\,dx\Bigr)\dfrac{d\alpha_i}{dt}+\Bigl(\int_{\Omega}\bigl(\rho(\mathbf{v}\cdot\nabla\Theta_i)\cdot\Theta_j+2\mu(\rho) d(\Theta_i)\cdot d(\Theta_j\bigr)\,dx\Bigr)\alpha_i\\
&\qquad-\Bigl(\int_{\Omega}(\curl\Gamma_i\times \mathbf{H})\cdot\Theta_j\,dx\Bigr)\beta_i=0,\\
&\Bigl(\int_{\Omega}\Gamma_i\cdot\Gamma_j\,dx\Bigr)\dfrac{d\beta_i}{dt}+\Bigl(\int_{\Omega}(h\dfrac{\curl\Gamma_i\times\mathbf{H}}{\rho}+\dfrac{\curl\Gamma_i}{\sigma(\rho)})\cdot\curl\Gamma_j\,dx\Bigr)\beta_i\\
&\qquad+\Bigl(\int_{\Omega}(\mathbf{H}\times\Theta_i)\cdot\curl\Gamma_j\,dx\Bigr)\alpha_i=0,
\end{aligned}
\right.
\end{equation}
with initial data for $i=1,\dots,n,$ defined by
\begin{equation}\label{dataODEs}
\left\{
\begin{aligned}
&\alpha_i (0)=\int_{\Omega} \mathbf{u}_{0}\cdot\Theta_i\,dx,\\
&\beta_i (0)=\int_{\Omega} \mathbf{B}_{0}\cdot\Gamma_i\,dx.
\end{aligned}
\right.
\end{equation}
Since the family $(\sqrt{\rho}\Theta_i)_{i=1,...n}$ (resp. $(\sqrt{\rho}\Gamma_i)_{i=1,...n}$) with $\rho\geq\underline{\rho}$ is free, it follows that the matrix $\bigl(\int_{\Omega}\rho\Theta_i\cdot\Theta_j\,dx\bigr)_{n\times n}$ (resp. $\bigl(\int_{\Omega}\rho\Gamma_i\cdot\Gamma_j\,dx\bigr)_{n\times n}$) is nonsingular for any $t\in[0, T].$ We note that the coefficients lie in the above ODEs are continuous on $[0, T]$ from assumptions \eqref{fixedrho}-\eqref{lb}. Thanks to the Cauchy-Lipschitz theorem, $(\alpha(t), \beta(t))$ exists uniquely and is continuous on $[0, T],$ where $\alpha(t)=(\alpha_1(t),\dots,\alpha_n(t))^T\in\R^n,$ $\beta(t)=(\beta_1(t),\dots,\beta_n(t))^T\in\R^n.$ Thus, we obtain with \eqref{ODEs} and \eqref{lub} 
\begin{equation*}
(\mathbf{u}, \mathbf{B})\in \mathcal{C}^1([0, T], \mathcal{V}_n)\times\mathcal{C}^1([0, T], \mathcal{W}_n)
\end{equation*}
In view of this regularity, we multiply the first equations in \eqref{ODEs} by $\alpha_j$ and second equations by $\beta_j,$ then we sum them  for $i, j=1,\dots,n.$ This yields
\begin{equation*}
\int_{\Omega}\rho(\partial_t\mathbf{u}+\mathbf{v}\cdot\nabla\mathbf{u})\cdot\mathbf{u}\,dx+\int_{\Omega}2\mu(\rho) |d(\mathbf{u})|^2\,dx-\int_{\Omega}(\curl\mathbf{B}\times\mathbf{H})\cdot\mathbf{u}\,dx=0,
\end{equation*}
and
\begin{equation*}
\int_{\Omega}\partial_t\mathbf{B}\cdot\mathbf{B}\,dx+ \int_{\Omega}\dfrac{|\curl\mathbf{B}|^2}{\sigma(\rho)}\,dx+\int_{\Omega}(\mathbf{H}\times\mathbf{u})\cdot\curl\mathbf{B}\,dx=0.
\end{equation*}
By integration by parts with \eqref{lv},  we get \eqref{en1}.
This completes the proof of Proposition \ref{P.l}.
\end{proof}

\subsection{Second step : a regularized approximation problem}\label{sec2.2}
In this step, we solve a non-linear regularized approximation problem by using the Schauder's fixed-point theorem (see Theorem II.3.9 in \cite{BF13}) and Proposition \ref{P.l}.

The following definition of mollifier can be found in e.g.  \cite{LE98} with several  properties. 
Define
$$\Omega_{\epsilon}:=\{x\in \Omega~|~{\rm{dist}}(x, \partial\Omega)>\epsilon\},$$
and $\eta\in\mathcal{C}^\infty(\mathbb{R}^{3})$ the standard mollifier by 
$$\eta(x)=\left\{
\begin{aligned}
&C\exp(\frac{1}{|x|^2-1})\hspace{0.4cm}\quad{\rm{if}}~~|x|<1\\
&0\hspace{2.64cm}\quad{\rm{if}}~~|x|\geq1,
\end{aligned}
\right.$$
for some normalizing constant $C$ such that $\int_{\mathbb{R}^{3}} \eta\,dx=1.$ 
Let $\epsilon\in(0, 1),$ define
$$\eta_{\epsilon}(x):=\frac{1}{\epsilon^3}\,\eta(\frac{x}{\epsilon}).$$
Now, we are ready to define approximate initial data. Since we  assumed our initial density to be bounded below by $\underline{\rho},$ we set
$$\widetilde{\rho}=\left\{
\begin{aligned}
&\rho_0\hspace*{0.58cm}~~{\rm{in}}~~\Omega\\
&\underline{\rho}\hspace*{0.75cm}~~{\rm{in}}~~\R^3\backslash\Omega,
\end{aligned}
\right.$$
and
$$\rho_0^\epsilon={\widetilde\rho_0}*\eta_\epsilon.$$
Then the initial density for the approximate
 system is defined by
 \begin{equation}\label{inapdensity}
 \rho_{n}|_{t=0}=\rho_{0}^{\epsilon}.
 \end{equation}
Thanks to assumption \eqref{inassum-rho}, it is clear for some universal constant $C_0$ independent of $\epsilon,$ we have
\begin{equation}\label{B-initial-appdensity}
{\underline{\rho}}\leq\rho^{\epsilon}_{0}\leq C_0,
\end{equation}
and $\rho_{0}^{\epsilon}\in\mathcal{C}^\infty(\overline\Omega)$
\begin{equation}\label{pinidensity}
\lim_{\epsilon\to 0} \rho^{\epsilon}_{0}=\rho_{0}~{\rm{in}}~ L^p{(\Omega)}~~~(1\leq p<\infty).
\end{equation}
We would also like to regularize $\mu(\xi),~\sigma(\xi)$ like in \cite{PL96}. Assume $\mu_{\epsilon}(\xi)$ is a $\mathcal{C}([0, \infty))$ function bounded away from zero, which is constant for $\xi\geq0$ large and such that $\sup_{[0, \infty)}|\mu_{\epsilon}-\mu|<\epsilon.$ We set
$$\widetilde{\mu_\epsilon(\rho)}=\left\{
\begin{aligned}
&\mu_{\epsilon}{(\rho)}\hspace*{0.5cm}~~{\rm{in}}~~\Omega\\
&1\hspace*{1.15cm}~~{\rm{in}}~~\R^3\backslash\Omega,
\end{aligned}
\right.$$
and define $\mu^\epsilon=\widetilde{\mu_\epsilon(\rho)}*\eta_{\epsilon}|_{\Omega}.$ We define $\sigma^\epsilon$ from $\sigma$ like $\mu^\epsilon$ from $\mu.$

% consider the initial condition for the velocity field $\mathbf{u}.$ First, we 

%\begin{equation*}
%({\rho_0}^\frac12)_{\epsilon}=({\tilde\rho_0}^\frac12*\eta_{\epsilon})|_{\Omega},
%\end{equation*}
%and define
%\begin{equation*}
%\mathbf{m}_0^\epsilon=\mathbf{m}_0{\rho_0}^{-\frac12}({\rho_0}^\frac12)_{\epsilon}.
%\end{equation*}
%Obviously,
%\begin{equation*}\label{iniapm}
%\mathbf{m}_0^\epsilon\to\mathbf{m}_0~{\rm{in}}~ L^2{(\Omega)},
%\end{equation*}
%and
%\begin{equation}\label{iniapmo}
%\mathbf{m}_0^\epsilon(\rho_{o}^{\epsilon})^{-\frac12}\to\mathbf{m}_0{\rho_0}^{-\frac12}~{\rm{in}}~ L^2{(\Omega)},
%\end{equation}
%From Lemma 2.1 in \cite{PL96} we have
%\begin{equation*}
%\mathbf{m}_0^{\epsilon}=\rho_{0}^{\epsilon}\mathbf{u}_{0}^{\epsilon}+\mathbf{q}_{0}^{ \epsilon},
%\end{equation*}
%where 
%\begin{align*}
%\mathbf{u}_{0}^{\epsilon},~\mathbf{q}_{0}^{\epsilon}\in L^2(\Omega),~{\rm{with}}~\curl\mathbf{q}_{0}^{\epsilon}=\div\mathbf{u}_{0}^{\epsilon}=0~{\rm{in}}\quad\mathcal{D}'(\Omega).
%\end{align*}
%Now we impose 
We set the initial condition for approximate velocity field and magnetic field  as
\begin{equation}\label{inapvelocity}
\mathbf{u}_{n}|_{t=0}=\mathbf{u}_{0,n}^\epsilon=P_{\mathcal{V}_n}\mathbf{u}_{0}^{\epsilon}
\end{equation}
and
\begin{equation}\label{inapmagnetic}
\mathbf{B}_{n}|_{t=0}=\mathbf{B}_{0, n}^{\epsilon}=P_{\mathcal{W}_n}\mathbf{B}_0^\epsilon,
\end{equation}
where $$\mathbf{u}_0^\epsilon=((\mathbf{u}_{0}\,1_{\Omega_\epsilon})*\eta_{\epsilon}),\qquad \mathbf{B}_0^\epsilon=((\mathbf{B}_{0}\,1_{\Omega_\epsilon})*\eta_{\epsilon})$$
and $P_{\mathcal{V}_n}$ (resp. $P_{\mathcal{W}_n}$) is the orthogonal projection in $L^2(\Omega)$ onto $\mathcal{V}_n$ (resp. $\mathcal{W}_n$).

Fixed $\epsilon,$ our approximation problem is stated as follows for $n\in\mathbb{N}^{*}.$ 
\begin{definition}\label{defap}
For any given $T>0,$
we say $(\rho_n, \mathbf{u}_n, \mathbf{B}_{n}),$ with
\begin{align*}
\rho_{n}\in\mathcal{C}([0, T]\times\Omega),\quad\mathbf{u}_n\in\mathcal{C}([0, T]; \mathcal{V}_n),\quad\mathbf{B}_n\in\mathcal{C}([0, T]; \mathcal{W}_n),
\end{align*}
is a global weak solution of the following approximation problem 
\begin{align}
&\partial_t\rho_n+\div(\rho_n\mathbf{u}_n)=0\label{ap1.1},\\
&\partial_t{(\rho_n\textbf{u}_n)}+\div(\rho\mathbf{u}_n\otimes\mathbf{u}_n)-\div(2\mu^\eps(\rho_n) d(\textbf{u}_n))+\nabla P_n= \curl \textbf{B}_n\times \textbf{B}_n\label{ap1.2},\\
&\div \mathbf{u}_n=0\label{apl.3},\\
&\partial_t{ \textbf{B}_n}+\curl(\textbf{B}_n\times \textbf{u}_n+h\dfrac{\curl \textbf{B}_n\times \textbf{B}_n}{\rho_n})=-\curl(\dfrac{\curl \textbf{B}_n}{\sigma^\epsilon(\rho_n)})\label{ap1.4}, \\
&\div\mathbf{B}_n=0,\label{ap1.5}
\end{align}
with the initial conditions 
\eqref{inapdensity}, \eqref{inapvelocity}, \eqref{inapmagnetic},  and boundary conditions
\begin{equation}\label{apbcu}
\mathbf{u}_{n}|_{\partial\Omega}=0,
\end{equation}
\begin{equation}\label{apbcb}
\left\{
\begin{aligned}
&\mathbf B_n\cdot\mathbf n=0\hspace*{4.35cm}~~{\rm{on}}~~\partial\Omega,\\
&(h\dfrac{\curl\textbf{B}_n\times\textbf{B}_n}{\rho_n}+\dfrac{\curl\textbf{B}_n}{\sigma^\epsilon(\rho_n)})\times \mathbf n=0 \hspace*{0.5cm}~~{\rm{on}}~~\partial\Omega,
\end{aligned}
\right.
\end{equation}
if \eqref{ap1.1}-\eqref{apbcb} are satisfied in the weak sense of Definition \ref{defweaksolu} with the test function spaces in \eqref {wfu} and \eqref{wfb} replaced by the restriction on $\mathcal{V}_n$ and $\mathcal{W}_n,$ respectively.
\end{definition}
With the above definition, we have the following existence result.
\begin{theorem}\label{Th_ap}
There exists a global weak solution $(\rho_n, \mathbf{u}_n, \mathbf{B}_n)$ to the above initial-boundary value problem.
\end{theorem}
\begin{proof}
In order to solve this non-linear approximation problem by the Schauder's fixed-point theorem, we shall construct a operator $F_n : I_n\to I_n,$ where the convex set
\begin{equation*}
I_n :=\Bigl{\{}(\bar{\mathbf{u}}, \bar{\mathbf{B}})\in\mathcal{C}([0, T]; \mathcal{V}_n)\times\mathcal C([0, T]; \mathcal{W}_n);~\sup_{t\in[0, T]}\|(\bar{\mathbf{u}}, \bar{\mathbf{B}})\|_{L^2(\Omega)}\leq R_0\Bigr{\}},
\end{equation*}
$R_0$ is a constant to be determined later.
We denote the input of $F_n$ to be $(\bar{\mathbf{u}}_n, \bar{\mathbf{B}}_n),$
and the corresponding output $F_n(\bar{\mathbf{u}}_n, \bar{\mathbf{B}}_n)$ to be $({\mathbf{u}}_n, {\mathbf{B}}_n).$ 
Then we define our operator as follows.
At first, with the input $(\bar{\mathbf{u}}_n, \bar{\mathbf{B}}_n),$ we consider the following linear problem:
\begin{equation}\label{lap1.1}
\left\{
\begin{aligned}
&\partial_t(\rho_n)+\div(\rho_n\bar{\mathbf{u}}_n)=0,\\
&\rho_n|_{t=0}=\rho_0^\epsilon.
\end{aligned}
\right.
\end{equation}
This is a classical transport equation since, $\bar{\mathbf{u}}_n$ is regular, $\div\bar{\mathbf{u}}_n=0$ and vanishes near $\partial\Omega.$ Thus $\rho_n$ is uniquely given by
\begin{equation*}
\rho_{n}(t, x)=\rho_0^\epsilon(X(0; x, t)),\quad\forall~(t, x)\in[0, T]\times\bar\Omega,
\end{equation*}
where $X$ is the solution of the ODE
\begin{equation*}
\left\{
\begin{aligned}
&\dfrac{dX}{ds}=\bar{\mathbf{u}}_n(s, X),\\
&X(t; x, t)=x.
\end{aligned}
\right.
\end{equation*}
Obviously, we have 
\begin{equation}\label{bapdensity}
\underline{\rho}\leq\rho_n\leq C_0.
\end{equation}
Since $\rho_o^\epsilon$ is smooth, 
so $\rho_n\in\mathcal{C}^1([0, T]\times\bar\Omega)$ and is bounded in this space uniformly in $(\bar{\mathbf{u}}_n, \bar{\mathbf{B}}_n).$

Now,  we set  $\rho=\rho_n,$ $\mathbf{v}=\bar{\mathbf{u}}_n,$ $\mathbf{H}=\bar{\mathbf{B}}_n$
and replace $\mu$ by $\mu^\epsilon(\rho_n),$ $\sigma$ by $\sigma^\epsilon(\rho_n)$ and we invoke Proposition \ref{P.l} to define $(\mathbf{u}_n, \mathbf{B}_n)$ as the solution of:
\begin{equation}\label{lapwfu}
\int_{\Omega}\rho_n(\partial_t\mathbf{u}_n+\bar{\mathbf{u}}_n\cdot\nabla\mathbf{u}_n)\cdot\Phi\,dx+\int_{\Omega}2\mu^\eps(\rho_n) d(\mathbf{u}_n)\cdot d(\Phi)\,dx=\int_{\Omega}(\curl\mathbf{B}_n\times\bar{\mathbf{B}}_n)\cdot\Phi\,dx,
\end{equation}for any $\Phi\in\mathcal{V}_n,$ 
\begin{equation}\label{lapwfb}
\int_{\Omega}\partial_t\mathbf{B}_n\cdot\Psi\,dx+\int_{\Omega}\bigl(\bar{\mathbf{B}}_n\times\mathbf{u}_n+h\dfrac{\curl\mathbf{B}_n\times\bar{\mathbf{B}}_n}{\rho_n}+\dfrac{\curl\mathbf{B}_n}{\sigma^\epsilon(\rho_n)}\bigr)\cdot\curl\Psi\,dx=0,
\end{equation}for any $\Psi\in\mathcal{W}_n,$
while $\mathbf{u}_n$ and $\mathbf{B}_n$ satisfy initial condition \eqref{inapvelocity}-\eqref{inapmagnetic}.

Next, let us choose $R_0$ such that $(\mathbf{u}_n, \mathbf{B}_n)\in I_n.$  Thanks to  Proposition \ref{P.l},  $(\mathbf{u}_n, \mathbf{B}_n)$  belongs to $\mathcal{C}([0, T]; \mathcal{V}_n)\times\mathcal C([0, T]; \mathcal{W}_n)$ and satisfies
\begin{align}
\dfrac{d}{dt}\int_{\Omega}(\rho_n|\textbf{u}_n|^2+|\mathbf{B}_n|^2)\,dx+\int_{\Omega}\Bigl(4{\mu}^\epsilon(\rho_n)|d(\textbf{u}_n)|^2+2\dfrac{|\curl\textbf{B}_n|^2}{\sigma^\epsilon(\rho_n)}\Bigr)\,dx=0\label{apen}.
\end{align}
This with \eqref{bapdensity} and $\div \mathbf{u}_n=0$ lead to
\begin{equation}\label{bapub}
\sup_{t\in[0, T]}\|\mathbf{u}_n\|_{L^2(\Omega)}+\sup_{t\in[0, T]}\|\mathbf{B}_n\|_{L^2(\Omega)}+\|\mathbf{u}_n\|_{L^2(0, T; V_1)}+\|\mathbf{B}_n\|_{L^2(0, T; W_1)}\leq C_1
\end{equation}
where $C_1$ is a constant independent of $R_0, \bar{\mathbf{u}}_{n}, \bar{\mathbf{B}}_n.$ Hence, by taking $R_0=C_1,$ we have $(\mathbf{u}_n, \mathbf{B}_n)\in I_n.$

Then we prove the compactness of mapping $F_n.$ In fact, with uniform bound \eqref{bapub} in hand, in view of the famous Aubin-Lions Lemma (see Theorem II. 5.16 in  \cite{BF13}), we only need to show some uniform bounds for  $\partial_t\mathbf{u}_n$ and $\partial_t\mathbf{B}_n$ in suitable spaces.
Thanks to Proposition \ref{P.l} again, we know  that actually $(\partial_t\mathbf{u}_n, \partial_t\mathbf{B}_n)$ belongs to $\mathcal{C}([0, T]; \mathcal{V}_n)\times\mathcal C([0, T]; \mathcal{W}_n)$. Hence by taking $\Phi=\partial_t\mathbf{u}_n,~\Psi=\partial_t\mathbf{B}_n$ in \eqref{lapwfu}-\eqref{lapwfb}, we have
\begin{align*}
&\int_{\Omega}(\rho_n|\partial_t\mathbf{u}_n|^2+|\partial_t\mathbf{B}_n|^2)\,dx\notag\\
&\quad=-\int_{\Omega}\bigl((\rho_n\bar{\mathbf{u}}_n\cdot\nabla\mathbf u_n)\cdot\partial_t\mathbf{u}_n+2\mu^\epsilon(\rho_n) d(\mathbf{u}_n)\cdot d(\mathbf{\partial_t\mathbf{u}_n})\bigr)\,dx\notag\\
&\quad\quad-\int_{\Omega}\Bigl(\bar{\mathbf{B}}_n\times\mathbf{u}_n+h\dfrac{\curl\mathbf{B}_n\times\bar{\mathbf{B}}_n}{\rho_n}+\dfrac{\curl\mathbf{B}_n}{\sigma^\epsilon(\rho_n)}\Bigr)\cdot\curl(\partial_t\mathbf{B}_n)\,dx\notag\\
&\quad\quad+\int_{\Omega}(\curl\mathbf{B}_n\times\bar{\mathbf{B}}_n)\cdot\partial_t\mathbf{u}_n\,dx.
\end{align*}
Since all norms in a finite dimensional space are equivalent and thanks to \eqref{bapub}, we obtain by H\"older's inequality
\begin{align}
&\underline{\rho}\|\partial_t\mathbf{u}_n\|_{L^2(\Omega)}^2+\|\partial_t\mathbf{B}_n\|_{L^2(\Omega)}^2\notag\\
\leq& C_0\|\bar{\mathbf{u}}_n\|_{L^\infty(\Omega)}\|\nabla\mathbf u_n\|_{L^2(\Omega)}\|\partial_t\mathbf{u}_n\|_{L^2(\Omega)}+2\bar{\mu}\|\nabla\mathbf{u}_n\|_{L^2(\Omega)}\| \nabla(\mathbf{\partial_t\mathbf{u}_n})\|_{L^2(\Omega)}\notag\\
&+\bigl(\|\bar{\mathbf{B}}_n\|_{L^\infty(\Omega)}\|\mathbf{u}_n\|_{L^2(\Omega)}+\frac{h}{\underline{\rho}}\|\curl\mathbf{B}_n\|_{L^2(\Omega)}\|\bar{\mathbf{B}}_n\|_{L^\infty(\Omega)}+\frac{1}{\underline{\sigma}}\|\curl\mathbf{B}_n\|_{L^2(\Omega)}\bigr)\notag\\
&\cdot\|\curl(\partial_t\mathbf{B}_n)\|_{L^2(\Omega)}+
\|\curl\mathbf{B}_n\|_{L^2(\Omega)}\|\bar{\mathbf{B}}_n\|_{L^\infty(\Omega)}\|\partial_t\mathbf{u}_n\|_{L^2(\Omega)}\notag\\
\leq& C\,(\|\partial_t\mathbf{u}_n\|_{L^2(\Omega)}+\|\partial_t\mathbf{B}_n\|_{L^2(\Omega)})\notag
\end{align}
and thus 
\begin{equation}\label{lapts}
\|\partial_t\mathbf{u}_n\|_{L^2(\Omega)}+\|\partial_t\mathbf{B}_n\|_{L^2(\Omega)}\leq C,
\end{equation}
where $C$ is a constant independent on $n, \epsilon.$
We conclude that $F_n$ is a compact operator on $I_n$ to itself.

In order to apply the Schauder's theorem, we still have to check the continuity of $F_n.$ It suffices to prove that the mapping is sequentially continuous. Let $\{(\bar{\mathbf{u}}_n^m, \bar{\mathbf{B}}_n^m)\}_{m\geq1}\subset I_n$ be a sequence which strongly converges to $(\bar{\mathbf{u}}_n, \bar{\mathbf{B}}_n)$ in $I_n.$ Recall our definition of mapping $F_n,$ we denote $\rho_n^m$ as the corresponding solution to \eqref{lap1.1} and let $(\mathbf{u}_n^m, \mathbf{B}_n^m)=F_n(\bar{\mathbf{u}}_n^m, \bar{\mathbf{B}}_n^m)$ as the corresponding solution to \eqref{lapwfu}-\eqref{lapwfb}.  It is clear when solving \eqref{lap1.1}, we have that $\{\rho_n^m\}_{m\geq1}$ is bounded in $\mathcal{C}^1([0, T]\times\bar\Omega)$ uniformly in $\{(\bar{\mathbf{u}}_n^m, \bar{\mathbf{B}}_n^m)\}_{m\geq1}.$ Thus Aubin-Lions Lemma implies that $\{\rho_n^m\}_{m\geq1}$ is pre-compact in $\mathcal{C}([0, T]\times\bar\Omega).$ For $\{(\mathbf{u}_n^m, \mathbf{B}_n^m)\}_{m\geq1},$ we know it is a subset of $I_n$ and the control of $\{(\partial_t\mathbf{u}_n^m, \partial_t\mathbf{B}_n^m)\}_{m\geq1}$ in $L^\infty(0, T; L^2(\Omega))$ can also be obtained by following a same procedure as to get \eqref{lapts}.
So one more application of the Aubin-Lions Lemma gives that $\{(\mathbf{u}_n^m, \mathbf{B}_n^m)\}_{m\geq1}$ is pre-compact in $\mathcal C ([0, T]; L^2(\Omega)).$ Without loss of generality, as we can always replace our original sequence by a weakly converging subsequence,   we conclude that $(\mathbf{u}_n^m, \mathbf{B}_n^m)$ converges to ${(\mathbf{u}_n, \mathbf{B}_n)}$ a solution of \eqref{lwfu}-\eqref{lwfb} when $m$ goes to $\infty.$  Since problem \eqref{lap1.1} and \eqref{lapwfu}-\eqref{lapwfb} are all linear problems and the solutions are shown to be unique, one has  ${(\mathbf{u}_n, \mathbf{B}_n)}=F_n(\bar{\mathbf{u}}_n, \bar{\mathbf{B}}_n),$ that is $(\mathbf{u}_n^m, \mathbf{B}_n^m)$ converges to $F_n(\bar{\mathbf{u}}_n, \bar{\mathbf{B}}_n)$ as $m\to\infty.$

From the Schauder's fixed-point theorem,  there exists a fixed point $(\mathbf{u}_n, \mathbf{B}_n)$ of $F_n$ in $I_n.$  It means that with input $(\mathbf{u}_n, \mathbf{B}_n)$ one can well-define $\rho_n$ as the solution of \eqref{ap1.1}. Moreover, following the definition of output,  by using \eqref{ap1.1} and the regularity of $\mathbf{u}_n, \rho_n,$ one can rewrite \eqref{lapwfu} as
\begin{align}
&\int_{\Omega}\partial_t(\rho_n\mathbf{u}_n)\cdot\Phi\,dx-\int_{\Omega}\rho_n\mathbf{u}_n\otimes{\mathbf{u}}_n\cdot\nabla\Phi\,dx+\int_{\Omega}2\mu^\eps(\rho_n) d(\mathbf{u}_n)\cdot d(\Phi)\,dx\notag\\
&\quad\quad=\int_{\Omega}(\curl\mathbf{B}_n\times{\mathbf{B}}_n)\cdot\Phi\,dx\label{apwfu},
\end{align}
and get
\begin{equation}\label{apwfb}
\int_{\Omega}\partial_t\mathbf{B}_n\cdot\Psi\,dx+\int_{\Omega}\bigl({\mathbf{B}}_n\times\mathbf{u}_n+h\dfrac{\curl\mathbf{B}_n\times{\mathbf{B}}_n}{\rho_n}+\dfrac{\curl\mathbf{B}_n}{\sigma^\epsilon(\rho_n)}\bigr)\cdot\curl\Psi\,dx=0,
\end{equation}

After time integration on \eqref{ap1.1}, \eqref{apwfu}, \eqref{apwfb} over $[0, T],$ we infer that $(\mathbf{u}_n, \mathbf{B}_n)$ satisfies the weak formulations in the Definition \ref{defap}.
Finally, we prove \eqref{meas}. Let $\gamma_m$ be a function of $\mathcal{C}^1(\R; \R).$ 
 Multiplying \eqref{ap1.1} by $\gamma^{'}_m(\rho_n)$ and using $\div \mathbf{u}_n=0$ we have
 \begin{align*}
\partial_t\gamma_m(\rho_n)+\mathbf{u}_n\cdot\nabla\gamma_m(\rho_n)=0.
 \end{align*}
We integrate this equation on $[0, T]\times\Omega$ and use again that $\div \mathbf{u}_n=0,$
with $\mathbf{u}_n$ vanishes near the boundary to obtain
\begin{align}
\int_{\Omega}\gamma_m(\rho_n(t, x))\,dx=\int_{\Omega}\gamma_m(\rho_{0}^\epsilon(x))\,dx\label{apmeas}.
 \end{align}
 For $0\leq\alpha\leq\beta<\infty$ we choose for $m$ large enough $0\leq\gamma_m\leq1$ such that
 \begin{equation*}
\left\{
\begin{aligned}
&\gamma_m(\lambda)=0 \quad{\rm{if}}~\lambda\notin[\alpha, \beta],\\
&\gamma_m(\lambda)=1\quad{\rm{if}}~\lambda\in[\alpha+\frac1m, \beta-\frac1m].
\end{aligned}
\right.
\end{equation*}
 Letting $m\to\infty$ in \eqref{apmeas} we deduce that \eqref{meas} holds for $\rho_n,$
 \begin{align*}
\int_{\Omega}1_{[\alpha, \beta]}(\rho_n(t, x))\,dx=\int_{\Omega}1_{[\alpha, \beta]}(\rho_{0}^\epsilon(x))\,dx
 \end{align*}
 where $1_{[\alpha, \beta]}(\lambda)$ is the characteristic function on $[\alpha, \beta].$
 This completes the proof of Theorem \ref{Th_ap}.
\end{proof}

\section{Convergence of the approximation problem}\label{sec.3}

The aim of this last section is to prove our main Theorem \ref{Th_1} by passing to the limit in the regularized approximation problem stated in Definition \ref{defap} as $n\to\infty$ and $\epsilon\to 0$. The fundamental tool is a compactness result due to P.-L. Lions \cite{PL96} that we recall here for its importance.
\begin{theorem}\label{Lions}
We suppose that two sequences $\rho_n$ and $\mathbf{u}_n$ are given satisfying $\rho_n\in\mathcal{C}([0, T]; L^1(\Omega)),$ $0\leq\rho_n\leq C$ a.e. on $(0, T)\times\Omega,$ $u_n\in L^2(0, T; H^1_0(\Omega)),$ $\|\mathbf{u}_n\|_{L^2(0, T; H^1_0(\Omega))}\leq C$ and $\div \mathbf{u}_n=0$ ($C$ denotes various constants independent of $n$). We  assume :
\begin{equation*}
\left\{
\begin{aligned}
&\partial_t\rho_n+\div(\rho_n{\mathbf{u}}_n)=0 \quad{\rm{in}}\quad\mathcal{D}^{'}((0, T)\times\Omega),\\
&\rho_n|_{t=0}=\rho_{0n},
\end{aligned}
\right.
\end{equation*}
and
\begin{align*}
&\rho_{0n}\to\rho_0\quad{\rm{in}}\quad L^1(\Omega),\\
&\mathbf{u}_n\rightharpoonup\mathbf{u}\quad{\rm{weakly}}\quad{\rm{in}}\quad L^2(0, T; H^1_0(\Omega)).
\end{align*}
Then :
\begin{enumerate}
\item $\rho_n$ converges in $\mathcal{C}([0, T]; L^p(\Omega))$ for all $1\leq p<\infty$ to the unique $\rho$ belonging to $\mathcal{C}([0, T]; L^1(\Omega))$ bounded on $(0, T)\times\Omega$ solution of 
\begin{equation*}
\left\{
\begin{aligned}
&\partial_t\rho+\div(\rho{\mathbf{u}})=0 \quad{\rm{in}}~~\mathcal{D}^{'}((0, T)\times\Omega),\\
&\rho|_{t=0}=\rho_0 ~~\qquad\qquad{\rm{a.e.~~ in}}~~\Omega.
\end{aligned}
\right.
\end{equation*}
\item We assume in addition that $\rho_n|\mathbf{u_n|^2}$ is bounded in $L^\infty(0, T; L^1(\Omega))$ and that we have for some $l\geq1$
$$|<\partial_t(\rho_n\mathbf{u}_n), \Phi>|\leq C\|\Phi\|_{L^2(0, T; H^l(\Omega))}$$
for all $\Phi\in\mathcal{D}((0, T)\times\Omega)$ such that $\div \Phi=0$ on $(0, T)\times\Omega.$ Then: 
\begin{align*}
&\sqrt{\rho_{n}}\mathbf{u}_n\to\sqrt{\rho}\mathbf{u}\quad{\rm{in}}\quad L^q(0, T; L^r(\Omega))\quad{\rm{for}}~2<q<\infty,~1\leq r<\frac{6q}{3q-4}\\
&\mathbf{u}_n\to\mathbf{u}\quad{\rm{in}}\quad L^\theta(0, T; L^{3\theta}(\Omega))\quad{\rm{for}}~1\leq\theta<2\quad{\rm{on~the~set}} ~\{(t, x)|~\rho(t, x)>0\}.
\end{align*}
\end{enumerate}
\end{theorem}

\subsection{Pass to the limit as \texorpdfstring{$n\to\infty$}{TEXT}}\label{sec3.1}

For fixed $\epsilon,$ we denote by $(\rho_n, \mathbf{u}_n, \mathbf{B}_n)$ the smooth approximate solution given by Theorem \ref{Th_ap}. It is clear that from \eqref{apen}, for any fixed $T>0,$ one has
\begin{align}
&\int_{\Omega}(\rho_n|\textbf{u}_n|^2\,dx+|\textbf{B}_n|^2)\,dx+\int_0^T\int_{\Omega}\Bigl(4\mu^\epsilon(\rho_n)|d(\textbf{u}_n)|^2+2\dfrac{|\curl\textbf{B}_n|^2}{\sigma^\epsilon(\rho_n)}\Bigr)\,dxdt\notag\\
&\quad\leq\int_{\Omega} (\rho_0^\epsilon|\mathbf{u}_{0}^\epsilon|^2+|\textbf{B}_0^\epsilon|^2)\,dx\label{apein}.
\end{align}
Note that $\rho_n$ is  bounded away from 0 uniformly, then we have the following bounds independent of $n$:
\begin{align}
&\|\rho_n|\mathbf{u}_n|^2\|_{L^\infty(0, T; L^1(\Omega))}\leq C\label{bm},\\
&\|\mathbf{u}_n\|_{L^\infty(0, T; L^2(\Omega))}\leq C,\\
&\|\mathbf{u}_n\|_{L^2(0, T; V_1)}\leq C\label{bu1},\\
&\|\mathbf{B}_n\|_{L^\infty(0, T; L^2(\Omega))}\leq C\label{bb},\\
&\|\mathbf{B}_n\|_{L^2(0, T; W_1)}\leq C\label{bb1},
\end{align}
which implies that as $n$ goes to infinity, up to extraction (we will extract subsequence if necessary),
\begin{align*}
&\mathbf{u}_n\rightharpoonup\mathbf{u}^{\epsilon}\quad{\rm{weakly}}\quad{\rm{in}}\quad L^2(0, T; V_1),\quad\mathbf{u}_n\rightharpoonup\mathbf{u}^{\epsilon}\quad{\rm{weakly^*}}\quad{\rm{in}}\quad L^\infty(0, T; V)\\
&\mathbf{B}_n\rightharpoonup\mathbf{B}^{\epsilon}\quad{\rm{weakly}}\quad{\rm{in}}\quad L^2(0, T; W_1), \quad\mathbf{B}_n\rightharpoonup\mathbf{B}^{\epsilon}\quad{\rm{weakly^*}}\quad{\rm{in}}\quad L^\infty(0, T; W).
\end{align*}

In view of \eqref{pinidensity}, \eqref{bapdensity} and \eqref{bu1}, the first assertion of Theorem \ref{Lions} implies that
\begin{equation}\label{cdensity}
\rho_n\to\rho^{\epsilon}\quad{\rm{in}}\quad\mathcal{C}([0, T]; L^p(\Omega))\quad{\rm {with} }\quad1\leq p<\infty
\end{equation}
and
$$\partial_t\rho^{\epsilon}+\div(\rho^{\epsilon}\mathbf{u^{\epsilon}})=0.$$
Our goal is now  to pass to the limit as $n$ goes to infinity in the following weak formulations  for approximation problem that have stated in Definition \ref{defap} :
\begin{align}
&\int_{0}^T\int_{\Omega}\bigl(-\rho_n\mathbf{u}_n\cdot\partial_t\Phi\,dxdt-\rho_n\mathbf{u}_n\otimes{\mathbf{u}}_n\cdot\nabla\Phi+2\mu^\eps(\rho_n) d(\mathbf{u}_n)\cdot d(\Phi)\bigr)\,dx\notag\\
&\quad\quad=-\int_{0}^T\int_{\Omega}(\curl\mathbf{B}_n\times{\mathbf{B}}_n)\cdot\Phi\,dxdt+\int_{\Omega}\rho_{0}^\epsilon\mathbf{u}_{0, n}^\epsilon\cdot\Phi(0, x)\,dx\label{limwfu},
\end{align}
for any $\Phi\in\mathcal{C}^1([0, T]\times\mathcal{V}_n)$ with $\div\Phi=0$ and $\Phi(T, \cdot)=0,$
\begin{align}
&\int_{0}^T\int_{\Omega}\bigl(-\mathbf{B}_n\cdot\partial_t\Psi+\bigl({\mathbf{B}}_n\times\mathbf{u}_n+h\dfrac{\curl\mathbf{B}_n\times{\mathbf{B}}_n}{\rho_n}+\dfrac{\curl\mathbf{B}_n}{\sigma^\epsilon(\rho_n)}\bigr)\cdot\curl\Psi\bigr)\,dxdt\notag\\
&\quad=\int_{\Omega}\mathbf{B}_{0, n}^\epsilon\cdot\Psi(0, x)\,ds\label{limwfb},
\end{align}
for any $\Psi\in\mathcal{C}^1([0, T]\times\mathcal{W}_n)$ with $\Psi(T, \cdot)=0.$

It turns out that we need to apply the second part of Theorem \ref{Lions}  to pass to the limit.
Since we already get \eqref{bm}, let us estimate $<\partial_t(\rho_n\mathbf{u}_n), \Phi>$ from equality \eqref{apwfu} with $\Phi\in\mathcal{D}((0, T)\times\Omega),~\div\Phi=0.$

First,  by Sobolev embeddings : ${H}^1(\Omega)\hookrightarrow L^6(\Omega), ~{H}^\frac12(\Omega)\hookrightarrow L^3(\Omega),$ we have
\begin{align*}
&|\int_0^T\int_{\Omega}\rho_n\mathbf{u}_n\otimes\mathbf{u}_n\cdot\nabla\Phi\,dxdt|\\
\leq&\|\rho_n\|_{L^\infty((0, T)\times\Omega)}\|\mathbf{u}_n\|_{L^\infty_T (L^2)}\|\mathbf{u}_n\|_{L^2_T(L^6)}\|\nabla\Phi\|_{L^2_T(L^3)}\\
\leq& C_0C^2\|\Phi\|_{L^2_T({H}^\frac32)},
\end{align*}
 and by $\div\mathbf{u}_n=0,~\div\Phi=0,$
\begin{align*}
&|\int_0^T\int_{\Omega} 2\mu^\epsilon(\rho_n)d(\mathbf{u}_n)\cdot d(\Phi)\,dxdt|\\
\leq&C\|{\mu}^\epsilon\|_{L^\infty([0, T]\times\Omega)}\|d(\mathbf{u}_n)\|_{L^2_T(L^2)}\|d(\Phi)\|_{L^2_T(L^2)}\\
\leq& C^2\|\Phi\|_{L^2_T({H}^1)}.
\end{align*}
Using the inequality $\|\Phi\|_{L^\infty(\Omega)}\leq C\|\Phi\|_{H^{s}}$ for $s>\frac32,$ we have
\begin{align*}
&|\int_0^T\int_{\Omega} (\curl\mathbf{B}_n\times\mathbf{B}_n)\cdot \Phi\,dxdt|\\
\leq&\|\curl\mathbf{B}_n)\|_{L^2_T(L^2)}\|\mathbf{B}_n\|_{L^\infty_T(L^2)}\|\Phi\|_{L^2_T(L^\infty)}\\
\leq& C^2\|\Phi\|_{L^2_T({H}^s)}.
\end{align*}
Therefore for any $l>\frac32,$ the second assertion of Theorem \ref{Lions} with strong convergence of $\rho_n$ then imply that as $n\to \infty$
\begin{align}
&{\rho_{n}}\mathbf{u}_n\to{\rho^\epsilon}\mathbf{u}^\epsilon\quad{\rm{in}}\quad L^q(0, T; L^r(\Omega))\quad{\rm{for}}~2<q<\infty,~1\leq r<\frac{6q}{3q-4}\label{cmo}\\
&\mathbf{u}_n\to\mathbf{u}^\epsilon\quad{\rm{in}}\quad L^\theta(0, T; L^{3\theta}(\Omega))\quad{\rm{for}}~1\leq\theta<2\label{scu}.
\end{align}

Let us prove the strong convergence of $\mathbf{B}_n$ to $\mathbf{B}^\epsilon$ in $L^2(0, T; W)$ at this moment. Indeed, from formula \eqref{apwfb},  for any $\Psi\in L^4(0, T; H^2(\Omega))$, one has
\begin{align*}
&|<\partial_t\mathbf{B}_n\cdot\Psi>|\\
\leq& |\int_0^T\int_{\Omega}\bigl(\mathbf{B}_n\times\mathbf{u}_n+h\dfrac{\curl\mathbf{B}_n\times\mathbf{B}_n}{\rho_n}+\dfrac{\curl\mathbf{B}_n}{\sigma^\epsilon(\rho_n)}\bigr)\cdot\curl\Psi\,dxdt|\\
\leq&\|\mathbf{B}_n\|_{L^4_T(L^3)}\|\mathbf{u}_n\|_{L^2_T (L^6)}\|\nabla\Phi\|_{L^4_T(L^2)}+\frac{h}{\underline{\rho}}\|\curl\mathbf{B}_n\|_{L^2_T(L^2)}\|\mathbf{B}_n\|_{L^4_T(L^3)}\|\nabla\Phi\|_{L^4_T(L^6)}\\
&+\frac{1}{\underline{\sigma}}\|\curl\mathbf{B}_n\|_{L^2_T(L^2)}\|\nabla\Phi\|_{L^2_T(L^2)}\\
\leq&C^2\|\Phi\|_{L^4_T({H}^2)},
\end{align*}
and thus $\{\partial_t \mathbf{B}_n\}_{n\geq1}$ is bounded in $L^\frac43(0, T; H^{-2}).$ 
\begin{remark}\label{Re-low-bound-rho}
If there is no positive lower bound to initial density but only $\rho_0\geq0,$ for fixed $\epsilon$ we can define the initial condition for approximate density as in \cite{PL96, GB97}, where $\rho^n\geq\epsilon$ is constructed. However, when passing to the limit as $\epsilon\to0$ we will lost above uniform bound with respect to $\epsilon$ due to the appearance of Hall-effect term.  This is the reason why we need to assume $\inf\rho_0>0.$
\end{remark}

 Keeping \eqref{bb1} in mind, we infer from the  Aubin-Lions Lemma  that
\begin{equation}\label{cb}
\mathbf{B}_n\to\mathbf{B}\quad{\rm{in}}\quad L^2(0, T; L^2(\Omega)).
\end{equation}
The weak and strong convergences obtained for $\rho_n, \mathbf{u}_n$ and $\mathbf{B}_n$ enable us to pass to the limit in the weak formulations \eqref{limwfu}-\eqref{limwfb} like in \cite{GB97},  except for the Hall-term. To deal with it, we write
\begin{align*}
&\dfrac{\curl\mathbf{B}_n\times\mathbf{B}_n}{\rho_n}-\dfrac{\curl\mathbf{B}\times\mathbf{B}}{\rho}\\
=&\dfrac{\rho-\rho_n}{\rho_n\rho}({\curl\mathbf{B}_n\times\mathbf{B}_n})+\dfrac{\curl\mathbf{B}_n\times(\mathbf{B}_n-\mathbf B)}{{\rho}}+\dfrac{(\curl\mathbf{B}_n-\curl\mathbf B)\times\mathbf{B}}{{\rho}}.
\end{align*}
Thanks to \eqref{bb}, \eqref{bb1} and \eqref{cdensity}  we know that the first term strongly tends to $0$
in $L^1(0, T; L^\frac65(\Omega)),$ while the second one strongly tends to $0$ in $L^1(0, T; L^1(\Omega))$ due to \eqref{cb}. Finally, the third term weakly tends to $0$ in $L^1(0, T; L^\frac32(\Omega))$ since $\curl\mathbf{B}_n$ is weakly convergent to $\curl B$ in $L^2((0, T)\times \Omega)$ and $\mathbf{B}$ lies in $L^2(0, T; L^6(\Omega)),$ $\rho$ lies in $L^\infty((0, T)\times\Omega).$

For the initial values, by definition of $\mathbf{u}_{0, n}^\epsilon$ and $\mathbf{B}_{0, n}^\epsilon,$ we have
\begin{align*}
\int_{\Omega}\rho_{0}^\epsilon\mathbf{u}_{0, n}^\epsilon\cdot\Phi(0, x)\,dx\to\int_{\Omega}\rho_0^\epsilon\mathbf{u}_0^\epsilon\cdot\Phi(0, x)\,dx=\int_{\Omega}\mathbf{m}_0^\epsilon\cdot\Phi(0, x)\,dx,
\end{align*}
\begin{align*}
\int_{\Omega}\mathbf{B}_{0, n}^\epsilon\cdot\Psi(0, x)\,dx\to\int_{\Omega}\mathbf{B}_0^\epsilon\cdot\Phi(0, x)\,dx.
\end{align*}
In conclusion, passing to the limit in \eqref{limwfu}-\eqref{limwfb} and energy inequality \eqref{apein}, we have obtained the following result :
\begin{proposition}
For any $T>0,$ there is a solution $(\rho^\epsilon, \mathbf{u}^\epsilon, \mathbf{B}^\epsilon)$
which satisfies
\begin{align*}
&\partial_t\rho+\div(\rho\mathbf{u})=0,\\
&\partial_t{(\rho\textbf{u})}+\div(\rho\mathbf{u}\otimes\mathbf{u})-\div(2\mu^\eps(\rho) d(\textbf{u}))+\nabla P= \curl \textbf{B}\times \textbf{B},\\
&\div \mathbf{u}=0,\\
&\partial_t{ \textbf{B}}+\curl(\textbf{B}\times \textbf{u}+h\dfrac{\curl \textbf{B}\times \textbf{B}}{\rho})=-\curl(\dfrac{\curl \textbf{B}}{\sigma^\epsilon(\rho)}), \\
&\div\mathbf{B}=0,
\end{align*}
 in the sense of Definition \ref{defap} with the initial data 
\begin{align*}
\rho|_{t=0}=\rho_0^\epsilon,\quad\mathbf{u}|_{t=0}=\mathbf{u}_0^\epsilon\quad\mathbf{B}|_{t=0}=\mathbf{B}_0^\epsilon.
\end{align*}
Moreover, the solution satisfies the following energy inequality :
\begin{align}
&\int_{\Omega}(\rho^\epsilon|\textbf{u}^\epsilon|^2\,dx+|\textbf{B}^\epsilon|^2)\,dx+\int_0^T\int_{\Omega}\mu^\epsilon(\rho^\epsilon)|\nabla\textbf{u}^\epsilon|^2\,dxdt
+2\int_0^T\int_{\Omega}\dfrac{|\curl\textbf{B}^\epsilon|^2}{\sigma^\epsilon(\rho^\epsilon)}\,dxdt\notag\\
&\quad\leq\int_{\Omega} (\rho_0^\epsilon|\mathbf{u}_{0}^\epsilon|^2+|\textbf{B}_0^\epsilon|^2)\,dx\label{enineps},
\end{align}
and
 \begin{align}
\int_{\Omega}1_{[\alpha, \beta]}(\rho^\epsilon(t, x))\,dx=\int_{\Omega}1_{[\alpha, \beta]}(\rho_{0}^\epsilon(x))\,dx\label{apemeas}.
 \end{align}
\end{proposition}

\subsection{Pass to the limit as \texorpdfstring{$\epsilon\to0$}{TEXT}}
For this passage to the limit, there is no  additional difficulty compare to the previous step since $\rho^\epsilon$ is still bounded away from zero by $\underline{\rho}$ uniformly. 
In particular,  from energy inequality \eqref{enineps} we have as $\epsilon\to0,$
\begin{align*}
&\mathbf{u}^\epsilon\rightharpoonup\mathbf{u}\quad{\rm{weakly}}\quad{\rm{in}}\quad L^2(0, T; V_1),\quad\mathbf{u}^\epsilon\rightharpoonup\mathbf{u}\quad{\rm{weakly^*}}\quad{\rm{in}}\quad L^\infty(0, T; V)\\
&\mathbf{B}^\epsilon\rightharpoonup\mathbf{B}\quad{\rm{weakly}}\quad{\rm{in}}\quad L^2(0, T; W_1), \quad\mathbf{B}^\epsilon\rightharpoonup\mathbf{B}\quad{\rm{weakly^*}}\quad{\rm{in}}\quad L^\infty(0, T; W),
\end{align*}
and by Theorem \ref{Lions}
\begin{align}
&\rho^\epsilon\to\rho\quad{\rm{in}}\quad\mathcal{C}([0, T]; L^p(\Omega))\quad{\rm {with} }\quad1\leq p<\infty\label{convdensity},\\
&{\rho^\epsilon}\mathbf{u}^\epsilon\to{\rho}\mathbf{u}\quad{\rm{in}}\quad L^q(0, T; L^r(\Omega))\quad{\rm{for}}~2<q<\infty,~1\leq r<\frac{6q}{3q-4}\notag,\\
&\mathbf{u}^\epsilon\to\mathbf{u}\quad{\rm{in}}\quad L^\theta(0, T; L^{3\theta}(\Omega))\quad{\rm{for}}~1\leq\theta<2\notag.
\end{align}
Moreover, again by Aubin-Lions lemma,
\begin{equation*}
\mathbf{B}^\epsilon\to\mathbf{B}\quad{\rm{in}}\quad L^2(0, T; L^2(\Omega)).
\end{equation*}
Just remark that $\rho^\epsilon$ bounded below by $\underline{\rho}$ is essential for getting uniform bound of $\partial_t\mathbf{B}^\epsilon$ in $L^\frac43(0, T; H^{-2}).$

In view of the construction of $\mu^\epsilon,~\sigma^\epsilon,$  one has
\begin{align*}
\mu^\epsilon\to\mu\quad{\rm{in}}\quad\mathcal{C}([0, T]; L^p(\Omega)),\quad\forall~1\leq p\leq\infty,
\end{align*}
\begin{align*}
\sigma^\epsilon\to\sigma\quad{\rm{in}}\quad\mathcal{C}([0, T]; L^p(\Omega)),\quad\forall~1\leq p\leq\infty.
\end{align*}
Hence, the above convergence properties with convergence for initial values ensure the existence part of Theorem \ref{Th_1}.
We now consider the initial value of energy inequality \eqref{enineps}. We observe that since $\nabla\times\mathbf{q}_0^\epsilon=0,$ then
\begin{align*}
\int_{\Omega}\rho_0^\epsilon|\mathbf{u}_{0}^\epsilon|^2\,dx=&\int_{\Omega}\frac{|\mathbf{m}_0^\epsilon-\mathbf{q}_0^\epsilon|^2}{\rho_0^\epsilon}\,dx\\
=&\int_{\Omega}\Bigl(\frac{|\mathbf{m}_0^\epsilon|^2}{\rho_0^\epsilon}-\frac{|\mathbf{q}_0^\epsilon|^2}{\rho_0^\epsilon}-2\mathbf{u}_0^\epsilon\cdot\mathbf{q}_0^\epsilon\Bigr)\,dx\\
=&\int_{\Omega}\Bigl(\frac{|\mathbf{m}_0^\epsilon|^2}{\rho_0^\epsilon}-\frac{|\mathbf{q}_0^\epsilon|^2}{\rho_0^\epsilon}\Bigr)\,dx.
\end{align*}

Thanks to convergence property \eqref{iniapmo}, letting $\epsilon\to0$ in \eqref{enineps} leads to the energy inequality \eqref{en2}.
Look at \eqref{apemeas}, an $\epsilon/2$ argument with convergence properties \eqref{pinidensity}, \eqref{convdensity} for $\rho_0^\epsilon$ and $\rho_0^\epsilon$ then implies\eqref{meas}.

To prove that $\mathbf{B}\in\mathcal{C}_w([0, T]; L^2(\Omega)),$ one first need to notice that $\partial_t\mathbf{B}$ lies in $L^1(0, T; H^{-2}),$ in particular $\mathbf{B}$ is almost everywhere equal to a continuous function  from  $[0, T]$ into $H^{-2}(\Omega).$ Finally, $\mathbf{B}\in L^\infty(0, T; L^2)$ and $H^2(\Omega)$ is dense in $L^2(\Omega)$ imply that $\mathbf{B}$ is weakly continuous from $[0, T]$ into $L^2(\Omega).$ $\mathbf{u}\in\mathcal{C}_w([0, T]; L^2(\Omega))$ based on similar argument.

This concludes the proof of Theorem \ref{Th_1}.\hfill\quad$\square$

\section{Weak-strong uniqueness}
In this section, we prove a weak-strong uniqueness property for global weak solutions obtained from Theorem \ref{Th_1}. 
We first remark that, in view of regularities $(\hat\rho, \mathbf{\hat u}, \mathbf{\hat B})\in\mathcal{C}([0, T]\times\Omega)$ and \eqref{regularity-ws} i.e.
\begin{align*}
\nabla\hat\rho, ~\partial_t\mathbf{\hat u}, ~\partial_t\mathbf{\hat B}\in L^2(0, T; L^3(\Omega))\quad{\rm{and}}\quad\nabla\mathbf{\hat u}, \nabla\mathbf{\hat B}\in L^2(0, T; L^\infty(\Omega)),
\end{align*}
we  actually could take $\mathbf{\hat u}$ and $\mathbf{\hat B}$ as test functions in the weak formulations \eqref{wfu}, \eqref{wfb}, and then get the following equalities for all $t\in (0, T)$
\begin{align}
&\int_{\Omega}\rho\mathbf{u}\cdot\mathbf{\hat u}\,dx+2\int_0^t\int_{\Omega}\mu(\rho) d(\mathbf{u})\cdot d(\mathbf{\hat u})\,dxds-\int_0^t\int_{\Omega}(\curl\mathbf{B}\times\mathbf{B})\cdot\mathbf{\hat u}\,dxds\notag\\
&\quad=\int_{\Omega}\mathbf{m}_0\cdot\mathbf{\hat u}(0, x)\,dx+\int_0^t\int_{\Omega}\rho\mathbf{u}\cdot(\partial_s\mathbf{\hat u}+u\cdot\nabla\mathbf{\hat u})\,dxds\label{wse1},
\end{align}
\begin{align}
&\int_{\Omega}\mathbf{B}\cdot\mathbf{\hat B}\,dx+\int_0^t\int_{\Omega}\bigl(h\dfrac{\curl\mathbf{B}\times\mathbf{B}}{\rho}+\dfrac{\curl\mathbf{B}}{\sigma(\rho)})\cdot\curl\mathbf{\hat B}\bigr)\,dxds\notag\\
&\quad=\int_{\Omega}|\mathbf{B}_0|^2dx+\int_0^t\int_{\Omega}\mathbf{B}\cdot\partial_s\mathbf{\hat B}\,dxds-\int_0^t\int_{\Omega}(\curl\mathbf{\hat B}\times\mathbf{B})\cdot\mathbf{u}\,dxds\label{wse2}.
\end{align}
Above, we have used the following two vector identities
\begin{align*}
&(\rho\mathbf{u}\otimes\mathbf{u})\cdot\nabla\mathbf{\hat u}=\rho\mathbf{u}\cdot(\mathbf{u}\cdot\nabla\mathbf{\hat u}),\\
&(\mathbf{B}\times\mathbf{u})\cdot\curl\mathbf{\hat B}=(\curl\mathbf{\hat B}\times\mathbf{B})\cdot\mathbf{u}.
\end{align*}
Let us notice that there exists a gradient term $\nabla\hat P$ that belongs to $L^1(0, T; L^\infty(\Omega))$ $+  ~L^2(0, T; H^{-1}(\Omega))$ coupled with $(\hat\rho, \mathbf{\hat u}, \mathbf{\hat B}).$
Next we write
\begin{align}
&\rho(\partial_t{\mathbf{\hat u}}+\mathbf{u}\cdot\nabla\mathbf{\hat u})-\div(2\mu(\rho) d(\mathbf{\hat u}))+\nabla\hat P- \curl \mathbf{\hat B}\times \mathbf{\hat B}\notag\\
&\quad =(\rho-\hat\rho)(\partial_t{\mathbf{\hat u}}+\mathbf{\hat u}\cdot\nabla\mathbf{\hat u})+\rho(\mathbf{u}-\mathbf{\hat u})\cdot\nabla\mathbf{\hat u}-\div\bigl(2(\mu(\rho)-\mu(\hat \rho))d(\mathbf{\hat u})\bigr),\label{wse3}\\
&\partial_t{ \mathbf{\hat B}}+\curl(\mathbf{\hat B}\times \mathbf{\hat u}+{h}\dfrac{\curl \mathbf{\hat B}\times \mathbf{\hat B}}{\hat\rho})+\curl(\dfrac{\curl \mathbf{\hat B}}{\sigma(\hat\rho)})=0\label{wse4},
\end{align}
and if we multiply \eqref{wse3}, \eqref{wse4} by $\mathbf{u},~\mathbf{B},$ respectively, and integrate over $(0, t)\times\Omega,$ then we have by integrations by parts
\begin{align}
&\int_0^t\int_{\Omega}\rho\mathbf{u}\cdot(\partial_s\mathbf{\hat u}+u\cdot\nabla\mathbf{\hat u})\,dxds+2\int_0^t\int_{\Omega}\mu(\rho)d(\mathbf{\hat u})\cdot d(\mathbf{u})\,dxds\notag\\
&\quad=\int_0^t\int_{\Omega}\Bigl(\curl \mathbf{\hat B}\times \mathbf{\hat B}+(\rho-\hat\rho)(\partial_s{\mathbf{\hat u}}+\mathbf{\hat u}\cdot\nabla\mathbf{\hat u})+\rho(\mathbf{u}-\mathbf{\hat u})\cdot\nabla\mathbf{\hat u}\Bigr)\cdot\mathbf{u}\,dxds\notag\\
&\quad\quad+\int_0^t\int_{\Omega}2(\mu(\rho)-\mu(\hat\rho))d(\mathbf{\hat u})\cdot d(\mathbf{u})\,dxds\label{wse00},
\end{align}
and
\begin{align}
&\int_0^t\int_{\Omega}\partial_s\mathbf{\hat B}\cdot \mathbf{B}+\dfrac{\curl\mathbf{\hat B}}{\sigma(\rho)}\cdot\curl\mathbf{B}\,dxds\notag\\
&\quad=\int_0^t\int_{\Omega} \Bigl((\sigma(\hat\rho)-\sigma(\rho))\dfrac{\curl\mathbf{\hat B}} {\sigma(\rho)\sigma(\hat\rho)}\cdot\curl\mathbf{B}-(\mathbf{\hat B}\times \mathbf{\hat u}+{h}\dfrac{\curl \mathbf{\hat B}\times \mathbf{\hat B}}{\hat\rho})\cdot\curl\mathbf{B}\Bigr)\,dxds\label{wse000}.
\end{align}
Bringing \eqref{wse00}, \eqref{wse000} into \eqref{wse1}, \eqref{wse2}, we obtain
\begin{align}
&\int_{\Omega}(\rho\mathbf{u}\cdot\mathbf{\hat u}+\mathbf{B}\cdot\mathbf{\hat B})\,dx+\int_0^t\int_{\Omega}\Bigl(4\mu(\rho) d(\mathbf{u})\cdot d(\mathbf{\hat u})+2\dfrac{\curl\mathbf{\hat B}}{\sigma(\rho)}\cdot\curl\mathbf{B}\Bigr)\,dxds\\
&\quad=\int_{\Omega}(|\mathbf{m}_0|^2{\rho_0}+|\mathbf{B}_0|^2)\,dx\notag\\
&\quad\quad+\int_0^t\int_{\Omega}\Bigl((\curl\mathbf{B}-\curl\mathbf{\hat B})\times(\mathbf{B}-\mathbf{\hat B})\cdot\mathbf{\hat u}-(\curl\mathbf{\hat B}\times(\mathbf{B}-\mathbf{\hat B}))\cdot(\mathbf{u}-\mathbf{\hat u})\notag\\
&\quad\quad +h\dfrac{(\curl\mathbf{B}-\curl\mathbf{\hat B})\times(\mathbf{\hat B}-\mathbf{B})}{\hat\rho}\cdot\curl\mathbf{\hat B}\notag\\
&\quad\quad+h(\rho-\hat\rho)\dfrac{(\curl\mathbf{B}-\curl\mathbf{\hat B})\times\mathbf{B}}{\rho\hat\rho}\cdot\curl\mathbf{\hat B}\notag\\
&\quad\quad+(\sigma(\hat\rho)-\sigma(\rho))\dfrac{\curl\mathbf{\hat B}} {\sigma(\rho)\sigma(\hat\rho)}\cdot\curl\mathbf{B}\notag\\
&\quad\quad+\bigl((\rho-\hat\rho)(\partial_s{\mathbf{\hat u}}+\mathbf{\hat u}\cdot\nabla\mathbf{\hat u})+\rho(\mathbf{u}-\mathbf{\hat u})\cdot\nabla\mathbf{\hat u}\bigr)\cdot\mathbf{u}\notag\\
&\quad\quad+2(\mu(\rho)-\mu(\hat\rho))d(\mathbf{\hat u})\cdot d(\mathbf{u})\Bigr)\,dxds\label{000}.
\end{align}
Then we multiply \eqref{wse3}, \eqref{wse4} by $\mathbf{\hat u},~\mathbf{\hat B}$ respectively, and integrate over $(0, t)\times\Omega$ to find
\begin{align}
&\frac12\int_{\Omega}\rho|\mathbf{\hat u}|^2\,dx+2\int_0^t\int_{\Omega}\mu(\rho) | d(\mathbf{\hat u})|^2\,dxds\notag\\
&\quad= \frac12\int_{\Omega}\dfrac{|\mathbf{m_0}|^2}{\rho_0}\,dx+\int_0^t\int_{\Omega}\Bigl(\bigl(\curl \mathbf{\hat B}\times \mathbf{\hat B}+(\rho-\hat\rho)(\partial_s{\mathbf{\hat u}}+\mathbf{\hat u}\cdot\nabla\mathbf{\hat u})+\rho(\mathbf{u}-\mathbf{\hat u})\cdot\nabla\mathbf{\hat u}\bigr)\notag\\
&\quad\quad\cdot\mathbf{\hat u}+2(\mu(\rho)-\mu(\hat \rho)) | d(\mathbf{\hat u})|^2\Bigr)\,dxds,\label{wse5}\\
&\frac12\int_{\Omega}|\mathbf{\hat B}|^2\,dx+\int_0^t\int_{\Omega}\dfrac{| \curl\mathbf{\hat B}|^2}{\sigma{(\rho)}}\,dxds\notag\\
&\quad=\frac12\int_{\Omega}|\mathbf{B_0}|^2\,dx-\int_0^t\int_{\Omega}\Bigl((\curl\mathbf{\hat B}\times\mathbf{\hat B})\cdot\mathbf{\hat u}-(\sigma(\hat\rho)-\sigma(\rho))\dfrac{|\curl \mathbf{\hat B}|^2}{\sigma(\rho)\sigma(\hat\rho)}\Bigr)\,dxds\label{wse6}.
\end{align}
Finally, we add up \eqref{wse5}, \eqref{wse6} and energy inequality \eqref{en2}, and subtract \eqref{000},  thanks to \eqref{positivemu}, \eqref{positivesigma} and \eqref{meas}, we have
\begin{align}
&\frac12(\underline{\rho}\|\mathbf{u}-\mathbf{\hat u}\|_{L^2}^2+\|\mathbf{B}-\mathbf{\hat B}\|_{L^2}^2)+\int_0^t\bigl(\underline{\mu}\|\nabla\mathbf{u}-\nabla\mathbf{\hat u}\|_{L^2}^2+\frac{1}{\bar\sigma}{\|\curl\mathbf{B}- \curl\mathbf{\hat B}\|_{L^2}^2}\bigr)\,ds\notag\\
&\quad\leq\int_0^t\int_{\Omega}\Bigl((\curl\mathbf{\hat B}-\curl\mathbf{B})\times(\mathbf{B}-\mathbf{\hat B})\cdot\mathbf{\hat u}+(\curl\mathbf{\hat B}\times(\mathbf{B}-\mathbf{\hat B}))\cdot(\mathbf{u}-\mathbf{\hat u})\notag\\
&\quad\quad -h\dfrac{(\curl\mathbf{B}-\curl\mathbf{\hat B})\times(\mathbf{\hat B}-\mathbf{B})}{\hat\rho}\cdot\curl\mathbf{\hat B}\notag\\
&\quad\quad-h(\rho-\hat\rho)\dfrac{(\curl\mathbf{B}-\curl\mathbf{\hat B})\times\mathbf{B}}{\rho\hat\rho}\cdot\curl\mathbf{\hat B}\notag\\
&\quad\quad+(\sigma(\hat\rho)-\sigma(\rho))\dfrac{\curl\mathbf{\hat B}} {\sigma(\rho)\sigma(\hat\rho)}\cdot(\curl\mathbf{\hat B}-\curl\mathbf{B})\notag\\
&\quad\quad+\bigl((\rho-\hat\rho)(\partial_s{\mathbf{\hat u}}+\mathbf{\hat u}\cdot\nabla\mathbf{\hat u})+\rho(\mathbf{u}-\mathbf{\hat u})\cdot\nabla\mathbf{\hat u}\bigr)\cdot(\mathbf{\hat u}-\mathbf{u})\notag\\
&\quad\quad+2(\mu(\rho)-\mu(\hat\rho))d(\mathbf{\hat u})\cdot d(\mathbf{\hat u}-\mathbf{u})\Bigr)\,dxds\notag.
\end{align}
%where we also have used some similar arguments as to get \eqref{en3}.
Hence, by H{\"o}lder inequality,  we have
\begin{align}
&\frac12(\underline{\rho}\|\mathbf{u}-\mathbf{\hat u}\|_{L^2}^2+\|\mathbf{B}-\mathbf{\hat B}\|_{L^2}^2)+\int_0^t\bigl(\underline{\mu}\|\nabla\mathbf{u}-\nabla\mathbf{\hat u}\|_{L^2}^2+\frac{1}{\bar\sigma}{\|\curl\mathbf{B}- \curl\mathbf{\hat B}\|_{L^2}^2}\bigr)\,ds\notag\\
&\quad\leq C\int_0^t \Bigl( \|\mathbf{B}-\mathbf{\hat B}\|_{L^2} \|\curl\mathbf{B}-\curl\mathbf{\hat B}\|_{L^2}\|\curl\mathbf{\hat B}\|_{L^\infty}\|\frac{1}{\hat\rho}\|_{L^\infty}( 1+\|\rho\|_{L^\infty}+\|\hat\rho\|_{L^\infty})\notag\\
&\quad\quad+\|\mathbf{B}-\mathbf{\hat B}\|_{L^2} \|\curl\mathbf{B}-\curl\mathbf{\hat B}\|_{L^2}\|\mathbf{\hat u}\|_{L^\infty}+\|\mathbf{B}-\mathbf{\hat B}\|_{L^2} \|\mathbf{u}-\mathbf{\hat u}\|_{L^2}\|\curl\mathbf{\hat B}\|_{L^\infty}\notag\\
&\quad\quad+\|\rho-\hat\rho\|_{L^2}\|\mathbf{u}-\mathbf{\hat u}\|_{L^6}\bigl(\|\partial_s\mathbf{\hat u}\|_{L^3}+\|\mathbf{\hat u}\|_{L^3}\|\nabla\mathbf{\hat u}\|_{L^\infty}\bigr)+\|\rho\|_{L^\infty} \|\mathbf{u}-\mathbf{\hat u}\|_{L^2}^2\|\nabla\mathbf{\hat u}\|_{L^\infty}\notag\\
&\quad\quad+\|\rho-\hat\rho\|_{L^2}\|\curl\mathbf{B}-\curl\mathbf{\hat B}\|_{L^2}\|\curl\mathbf{\hat B}\|_{L^\infty}+\|\rho-\hat\rho\|_{L^2}\|\mathbf{u}-\mathbf{\hat u}\|_{L^2}\|\mathbf{\hat u}\|_{L^\infty}\notag\\
&\quad\quad+\|\rho-\hat\rho\|_{L^2}\|\curl\mathbf{B}-\curl\mathbf{\hat B}\|_{L^2}\|\frac{1}{\hat\rho}\|_{L^\infty}(1+\|\mathbf{\hat B}\|_{L^\infty}\|\curl\mathbf{\hat B}\|_{L^\infty})\Bigr)\,ds\label{wseub},
\end{align}
in which we have specially estimate
\begin{align*}
&\int_{\Omega} h(\rho-\hat\rho)\dfrac{(\curl\mathbf{B}-\curl\mathbf{\hat B})\times\mathbf{B}}{\rho\hat\rho}\cdot\curl\mathbf{\hat B}\,dx\\
&\quad=\int_{\Omega} h(\rho-\hat\rho)\bigl(\dfrac{(\curl\mathbf{B}-\curl\mathbf{\hat B})\times(\mathbf{B}-\mathbf{\hat B})}{\rho\hat\rho}+\dfrac{(\curl\mathbf{B}-\curl\mathbf{\hat B})\times\mathbf{\hat B}}{\rho\hat\rho}\bigr)\cdot\curl\mathbf{\hat B}\,dx\\
&\quad\leq h\|\frac{1}{\hat\rho}\|_{L^\infty}\|\frac{1}{\rho}\|_{L^\infty}\|\curl\mathbf{B}-\curl\mathbf{\hat B}\|_{L^2}\bigl((\|\rho\|_{L^\infty}+\|\hat\rho\|_{L^\infty})\|\mathbf{B}-\mathbf{\hat B}\|_{L^2}\\
&\quad\qquad\quad+\|\rho-\hat\rho\|_{L^2}\|\mathbf{\hat B}\|_{L^\infty}\|\curl\mathbf{\hat B}\|_{L^\infty}\bigr).
\end{align*}
Now, we apply Young's inequality and sometimes Sobolev embeddings to \eqref{wseub} and deduce from the assumptions on $(\hat\rho, \mathbf{\hat u}, \mathbf{\hat B})$ that for all $t\in(0, T)$
\begin{align*}
&\frac12(\underline{\rho}\|\mathbf{u}-\mathbf{\hat u}\|_{L^2}^2+\|\mathbf{B}-\mathbf{\hat B}\|_{L^2}^2)+\int_0^t\bigl(\underline{\mu}\|\nabla\mathbf{u}-\nabla\mathbf{\hat u}\|_{L^2}^2+\frac{1}{\bar\sigma}{\|\curl\mathbf{B}- \curl\mathbf{\hat B}\|_{L^2}^2}\bigr)\,ds\notag\\
&\quad\leq \int_0^t C(s)\bigl(\|\mathbf{u}-\mathbf{\hat u}\|_{L^2}^2+\|\mathbf{B}-\mathbf{\hat B}\|_{L^2}^2+\|\rho-\hat\rho\|_{L^2}^2\bigr)\,ds,
\end{align*}
where $C(s)$ denotes a various measurable function in $L^1(0, T).$ 

It remains to estimate $\|\rho-\hat\rho\|_{L^2}$ as in \cite{PL96}. In fact, we have
\begin{equation*}
\frac12\|\rho-\hat\rho\|_{L^2}^2\leq\int_0^t \|\mathbf{\hat u}-\mathbf{u}\|_{L^6}\|\nabla\hat\rho\|_{L^3}\|\rho-\hat\rho\|_{L^2}\,ds.
\end{equation*}
Combining above two estimates, we write
\begin{align*}
&\frac12(\underline{\rho}\|\mathbf{u}-\mathbf{\hat u}\|_{L^2}^2+\|\mathbf{B}-\mathbf{\hat B}\|_{L^2}^2+\|\rho-\hat\rho\|_{L^2}^2)\\
&\hspace*{2cm}+\int_0^t\bigl(\underline{\mu}\|\nabla\mathbf{u}-\nabla\mathbf{\hat u}\|_{L^2}^2+\frac{1}{\bar\sigma}{\|\curl\mathbf{B}- \curl\mathbf{\hat B}\|_{L^2}^2}\bigr)\,ds\notag\\
&\quad\leq \int_0^t C(s)\bigl(\|\mathbf{u}-\mathbf{\hat u}\|_{L^2}^2+\|\mathbf{B}-\mathbf{\hat B}\|_{L^2}^2+\|\rho-\hat\rho\|_{L^2}^2\bigr)\,ds,
\end{align*}
and conclude by applying Gr{\"o}nwall's lemma  that $(\hat\rho, \mathbf{\hat u}, \mathbf{\hat B})\equiv (\rho, \mathbf{ u}, \mathbf{ B})$ a.e. in $(0, T)\times\Omega.$ \hfill\quad$\square$\\
\medskip

%\textbf{Conflicts of interest.}
%On behalf of all authors, the corresponding author states that there is no conflict of interest. 

\textbf{Acknowledgment.} {The author would like to thank Prof. R{a}pha\"{e}l Danchin for his wise suggestions and remarks.
This research is partly funded by the B{\'e}zout Labex and ANR, reference ANR-10-LABX-58. }

\end{document}